\documentclass[11pt]{article}

\usepackage[T1]{fontenc}
\usepackage[utf8]{inputenc}
\usepackage{lmodern}
\usepackage{amsmath,amssymb,amsthm,mathtools}
\usepackage{mathrsfs}
\usepackage{enumitem}
\usepackage[margin=1in]{geometry}
\usepackage{microtype}
\usepackage{hyperref}
\usepackage{orcidlink}
\usepackage{verbatim}
\usepackage{makecell}
\usepackage{tikz-cd}

\hypersetup{colorlinks=true,linkcolor=blue,citecolor=blue,urlcolor=blue}

\numberwithin{equation}{section}

\newtheorem{theorem}{Theorem}[section]
\newtheorem{proposition}[theorem]{Proposition}
\newtheorem{corollary}[theorem]{Corollary}
\newtheorem{lemma}[theorem]{Lemma}
\newtheorem{definition}[theorem]{Definition}

\newtheorem{remark}[theorem]{Remark}

\newtheorem{construction}[theorem]{Construction}
\newtheorem{notation}[theorem]{Notation}

\DeclareMathOperator{\CAlg}{CAlg}

\newcommand{\R}{\mathbb R}

\title{Derived Azumaya Algebras on Compact Derived Smooth Manifolds}

\author{
  Yimu Mao\,\orcidlink{0009-0003-9288-6282}
  \and
  Christopher Tropp\,\orcidlink{0009-0004-0492-7476}
}

\date{}

\begin{document}

\maketitle

\begin{abstract}
We extend Toën's derived Brauer theory to the open site of Spivak's derived smooth manifolds. For a compact derived smooth manifold $X$, we define a spectralized structure sheaf $\mathbb O_X$ and prove that every Picard twist can be represented by a derived Azumaya algebra, whose Morita equivalence class is uniquely determined by the twist. In the open locally Morita trivial setting, this gives a natural bijection
\begin{align*}
\Phi_X:\mathbf{dBr}^{\mathrm{open}}(X)
\xrightarrow{\cong}
H^1(|X|;\underline{\mathbb Z})
\times
\pi_0\Gamma\bigl(|X|,\mathrm B^2GL_1(\mathbb O_X)\bigr).
\end{align*}

If we consider $C^\infty_X(\mathbb C)$ on ordinary smooth manifolds, the second factor is $H^3(|X|;\underline{\mathbb Z})$, which recovers the Dixmier–Douady invariant. If we consider $C^\infty_X(\mathbb R)$, the second factor is $H^2(|X|;\underline{\mathbb Z/2})$, which recovers the corresponding invariant of Donovan and Karoubi.
\end{abstract}

\tableofcontents

\section{Introduction}
The theory of Azumaya algebras provides a natural bridge between algebra, geometry, topology, and
cohomology. Over a commutative ring, an Azumaya algebra is a noncommutative algebra which
becomes a matrix algebra after a suitable faithfully flat base change. More generally, Azumaya
algebras on schemes are locally Morita equivalent to the structure sheaf. Their Morita-equivalence
classes form the Brauer group, which is related to degree-two cohomology with coefficients in the
sheaf of units.
\\
In 1966, Grothendieck developed Brauer theory for Azumaya algebras on schemes and related it to the second étale cohomology of the sheaf of units \cite{GrothendieckI}. In topology, Dixmier–Douady in 1963 used third integer cohomology to classify locally trivial bundles of compact operator $C^*$-algebras over paracompact spaces \cite{DixmierDouady}. In 1970, Donovan–Karoubi studied real and complex graded Brauer groups and applied them to $K$-theory with local coefficients \cite{DonovanKaroubi}. In 1996, Murray introduced bundle gerbes as geometric models for third integer cohomology and Dixmier–Douady classes \cite{Murray}. In the derived setting, in 2012, To\"en introduced derived Azumaya algebras and proved their cohomological classification on quasi-compact, quasi-separated schemes \cite{Toen}. This also gives algebraic representatives for non-torsion classes in second cohomology.
\\
We use the methods from Toën and Antieau–Gepner to study this problem for compact derived smooth manifolds in the sense of Spivak. Let $X$ be such a manifold and $\mathbb O_X$ its spectral structure sheaf. We prove that every Picard twist is represented by a sheaf of derived Azumaya algebras. There exist a natural bijection
\begin{align*}
\Phi_X: \mathbf{dBr}^{\mathrm{open}}(X) \xrightarrow{\;\cong\;} H^1(|X|; \underline{\mathbb Z}) \times \pi_0 \Gamma \bigl( |X|,\mathrm B^2 GL_1(\mathbb O_X) \bigr).
\end{align*}
Here, $\mathbf{dBr}^{\mathrm{open}}(X)$ be the set of Morita equivalence classes of sheaves of derived Azumaya algebras that are locally Morita equivalent to $\mathbb O_X$ on an open cover.

\section{Derived Smooth Manifolds and Spectralization}

\subsection{Compact Derived Smooth Manifolds}

\begin{definition}[Euclidean category]
 \emph{The Euclidean category $\mathbb{E}$ is defined as:}
\begin{itemize}
    \item   \emph{$\operatorname{Ob}(\mathbb{E})=\{\mathbb{R}^n|n \geq 0\}$;}
    \item  \emph{$\operatorname{Mor}_{\mathbb{E}}(\mathbb{R}^m,\mathbb{R}^n) = C^{\infty}(\mathbb{R}^m,\mathbb{R}^n)$;}
    \item  \emph{$\operatorname{comp}_{\mathbb{E}}(f,g)=g \circ f$ for $f: \mathbb{R}^n \rightarrow \mathbb{R}^m,\; g:\mathbb{R}^m \rightarrow \mathbb{R}^k$.}
\end{itemize}
\end{definition}

\begin{definition}[Lax simplicial $C^\infty$-rings]
\emph{For each $i\geq 0$, there is a functor
\begin{align*}
H_i:\mathbb E\longrightarrow\mathbf{sSet},
\qquad
H_i(\mathbb R^n)
=\operatorname{Hom}_{\mathbb E}(\mathbb R^i,\mathbb R^n),
\end{align*}
where $\mathbf{sSet}$ is the category of simplicial sets. There are maps $\mathbb R^{i+j}\to\mathbb R^i$ and $\mathbb R^{i+j}\to\mathbb R^j$ which induce this projection $p_{i,j}:H_i\amalg H_j\longrightarrow H_{i+j}$. We denote $P=\{p_{i,j}\mid i,j\geq 0\} \cup \{\varnothing \rightarrow H_0 \}$. The model category of lax simplicial $C^\infty$-rings is the left Bousfield localization
\begin{align*}
s\mathbf C^\infty:=L_P\bigl(\mathbf{sSet}^{\mathbb E}_{\mathrm{inj}}\bigr).
\end{align*}
Its fibrant objects are precisely the injectively fibrant functors
$F:\mathbb E\to\mathbf{sSet}$, such that $F(\mathbb R^0) \simeq *$ Induced by the projections, we can follow that for all $i,j\geq 0$ the following map is 
\begin{align}
\label{lax_SCinftyRing}
F(\mathbb R^{i+j})
\longrightarrow
F(\mathbb R^i)\times F(\mathbb R^j)
\end{align}
is a weak equivalence of simplicial sets. }
\end{definition}

\begin{remark}
     \emph{If we ask the map above \ref{lax_SCinftyRing} to be an isomorphism and $F(\mathbb{R}^0) = *$, then this is the strict simplicial $C^{\infty}$-ring. We denote it as $sC^{\infty}$Ring. There exists a Quillen equivalence between $s\mathbf C^\infty$ $\simeq$ $ sC^{\infty}$Ring, after localization to its own weak equivalence, we arrive at the following:
    \begin{align}
    \label{lax-strict}
       \sigma: s\mathbf C^\infty[W_{lax}^{-1}] \longrightarrow sC^{\infty}\mathrm{Ring}[W_{str}^{-1}].
    \end{align}}
    
     \emph{See more discussions in \cite[Definition~5.1] {Spivak1} and \cite[Remark~5.2]{Spivak1}}
\end{remark}

\begin{remark}
 \emph{There exists a functor (\cite[Remark 5.7]{Spivak1}) from $s\mathbf C^\infty$ to $s\mathbb{R}$ that denotes the simplicial $\mathbb{R}$-algebra. This preserves the following underlying $\mathbb{R}$-algebra functor
\begin{align}
\label{lax-forgetful}
U:s\mathbf C^\infty[W_{lax}^{-1}] \longrightarrow s\mathbb{R}[W_{\mathbb{R}}^{-1}].
\end{align}
We denote the strict simplicial commutative $\mathbb{R}$-algebras as $sCAlg_{\mathbb{R}}$. There also exists an algebralization forgetful functor (\cite{Steffens}) from $sC^\infty \mathrm{Ring}$ to $sCAlg_{\mathbb{R}}$, which induces the following map
\begin{align}
\label{strict-forgetful}
(-)^{alg}:sC^\infty\mathrm{Ring}[W_{\mathrm{str}}^{-1}] \longrightarrow s\mathrm{CAlg}_{\mathbb R}[W_{\mathrm{alg}}^{-1}].
\end{align}}
 \emph{The algebralization forgetful functor (\cite{Steffens}) is actually the $\infty$-category version of the mentioned strict simplicial $C^{\infty}$ ring after localization. Therefore, this can also be applied onto strict simplicial $C^{\infty}$ rings.}

\end{remark}

\begin{definition}[Affine Derived Smooth Manifolds] \label{adsm}
 \emph{An affine derived manifold is a pair $\mathcal{X} = (X, \mathcal{O}_X) \in \mathbf{LRS}$, where $\mathcal{O}_X \in \mathrm{Shv}(|X|,sC^{\infty}Ring)$ is fibrant, and which can be obtained as the homotopy limit in the simplicial category $\mathbf{LRS}$ of local $C^\infty$-ringed spaces in a diagram of the form
}
\[
\begin{array}{ccc}
\mathcal{X}
& \longrightarrow &
\mathbb{R}^{0}
\\[6pt]
\Big\downarrow\! g
&&
\Big\downarrow\! 0
\\[6pt]
\mathbb{R}^{n}
& \xrightarrow{\;f\;} &
\mathbb{R}^{k}
\end{array}
\]
\end{definition}

Here, $\mathbb R^0\to\mathbb R^k$ is the origin and $f:\mathbb R^n\to\mathbb R^k$ is a smooth map. Recall that $\mathbf{LRS}$ denotes Spivak's simplicial category of local $C^{\infty}$-ringed spaces \cite[Definition~6.3]{Spivak1}; homotopy limits are understood after the usual cofibrant-fibrant replacement convention \cite[Notation~2.1]{Spivak1}. We sometimes refer to an affine derived manifold as a local model for derived manifolds. The map $g:\mathcal{X} \to \mathbb{R}^n$ is called the canonical inclusion of the zero set.

\begin{definition}[Derived Smooth Manifolds]\label{cdsm_def}
 \emph{A Compact Derived Smooth Manifold in the sense of Spivak is a ringed Space $\mathcal{X} = (|X|,\mathcal{O}_{X})$ with the following properties:}
\begin{itemize}
    \item  \emph{$|X|$ is the underlying compact Hausdorff topological space}
    \item  \emph{The structure sheaf 
    \[\mathcal{O}_{X}: \operatorname{Open(|X|)}^{\operatorname{op}} \longrightarrow s\mathbf C^\infty \]
    is a fibrant object of the model category $\mathrm{Shv}(|X|,s\mathbf C^\infty)$}
    \item  \emph{For every $x \in |X|$, there is a local ring $\pi_0(\mathcal{O}_{X,x})$} 
    \item  \emph{there is also an open covering $\bigcup_{i}U_{i} = |X|$, such that every $(U_i,\mathcal{O}_{X}|_{U_i})$ is an affine derived manifold (See Definition \ref{adsm}).}
\end{itemize}
    
\end{definition}

\begin{remark}
     \emph{There exists another module called "Derived Manifold" which is from D. Carchedi and P. Steffens \cite{CarchediSteffens}. Here, the "Derived Smooth Manifold" from D. Spivak is a special case of "Derived Manifold" from D. Carchedi and P. Steffens. See more discussion in Section 5.2 \cite{CarchediSteffens}}.
\end{remark}

\subsection{Spectralization of the Structure Sheaf}
\begin{construction}
 \emph{We can do the following construction:}
\begin{center}
\begin{tikzcd}[column sep=large, row sep=large]
s\mathbf C^\infty[W_{\mathrm{lax}}^{-1}]
  \arrow[r, "\sigma"]
  \arrow[d, "U"']
&
sC^\infty\mathrm{Ring}[W_{\mathrm{str}}^{-1}]
  \arrow[d, "(-)^{\mathrm{alg}}"]
  \arrow[r, "\lambda=\psi\circ(-)^{\mathrm{alg}}"]
&
\mathrm{cdga}_{\mathbb R}^{\geq 0}[W_{\mathrm{dg}}^{-1}]
  \arrow[r, "\phi"]
&
\mathrm{CAlg}_{H\mathbb R}^{\mathrm{cn}} 
\\
s\mathbb R[W_{\mathbb R}^{-1}]
  \arrow[r, "\delta"]
&
s\mathrm{CAlg}_{\mathbb R}[W_{\mathrm{alg}}^{-1}]
  \arrow[ur, "\psi"']
\end{tikzcd}
\end{center}
\end{construction}

Functors  $\sigma$, $U$, $(-)^{alg}$ are respectively mentioned in maps \ref{lax-strict}, \ref{lax-forgetful} and \ref{strict-forgetful}. Functors $\delta$ and $\psi$ are mentioned in \cite[Proposition 5.9]{Spivak1}, they both provide a Quillen equivalence. Over the characteristic-zero field $\R$, there are two Quillen equivalences, one is between the lax simplicial $\mathbb{R}$-algebra and the strict simplicial commutative $\mathbb{R}$-algebra, another one is between simplicial commutative $\R$-algebras and connective commutative differential graded $\R$-algebras \cite[Proposition~5.9]{Spivak1}.  A Quillen equivalence induces a
Dwyer--Kan equivalence of simplicial localizations
\cite[Proposition~5.4]{DwyerKan}. Hence, it induces an equivalence of the
localized $\infty$-categories
\begin{equation}\label{eq:Psi}
 \psi\colon
 s\CAlg_{\R}[W_{alg}^{-1}]
 \xrightarrow{\ \simeq\ }
 \operatorname{cdga}^{\geq0}_{\R}[W_{\mathrm{dg}}^{-1}].
\end{equation}
Here, $W_{alg}$ denotes the simplicial weak equivalences and $W_{dg}$
the quasi-isomorphisms.

According to \cite[Proposition 7.1.4.11]{LurieHA}, this proposition gives the equivalence between all $cdga$ and $E_{\infty}$-$HR$-algebras, if we choose $R = \mathbb{R}$ and restrict to the full subcategory of connective objects, 
we can directly get the functor $\phi$.

\begin{proposition}\label{prop:sp-preservation}
\emph{The spectralization functor $\operatorname{sp} := \phi\circ\psi\circ\delta\circ U$ preserves all limits and filtered colimits.}
\end{proposition}

\begin{proof}
The functor $U$ is the right derived functor of a right Quillen
functor and therefore preserves all limits. Since $\delta$, $\psi$, and $\phi$ are equivalences of $\infty$-categories, they also preserve all limits. That's why $sp$ preserves all limits.

$U$ is based on a restriction that is along $\mathbb E^{\mathrm{alg}}\hookrightarrow\mathbb E$. The computation of filtered colimits of weakly product-preserving functors is objectwise, because filtered colimits in spaces commute with finite products. So $U$ preserves filtered colimits. Since $\delta$, $\psi$, and $\phi$ all preserve  colimits, the spectralization functor $sp$ preserves filtered colimits.
\end{proof}

\subsection{Spectralized Structure Sheaf}
\begin{definition}[The spectralized structure presheaf]
 \emph{Let $\mathcal X=(|X|,\mathcal O_X)$ be a compact derived smooth manifold, and let $\mathcal O_{X,\infty }$ denote its structure sheaf regarded as a functor to the associated $\infty$-category. Now we define the spectralized structure presheaf
\begin{align}
\mathbb O_X^{\mathrm{pre}}:
\operatorname{Open}(|X|)^{\mathrm{op}}
\longrightarrow
\mathrm{CAlg}_{H\mathbb R}^{\mathrm{cn}}
\end{align}
by
\[
\mathbb O_X^{\mathrm{pre}}(V)
:=
\operatorname{sp}\bigl(\mathcal O_{X,\infty}(V)\bigr)= \phi\circ\psi\circ\delta\circ U\bigl(\mathcal O_{X,\infty}(V)\bigr).\]}
\end{definition}

\begin{proposition}
\label{O_X-hypersheaf}
\emph{The above presheaf $\mathbb{O}_X^{pre}$ satisfies hyperdescent in $\operatorname{CAlg}_{H\mathbb R}$, and its underlying spectral presheaf is a hypersheaf.}
\end{proposition}

\begin{proof}

According to the definition of the model category of sheaves of simplicial $C^\infty$-rings and of a local $C^\infty$-ringed space \cite[Definitions~6.2 and~6.3]{Spivak1}, we know that the structure sheaf $\mathcal O_X$ is fibrant in the model category
obtained by localization at hypercovers. Therefore,
$\mathcal O_{X,\infty}$ satisfies hyperdescent.

Let $V_\bullet\longrightarrow V$ be a hypercover of an open subset $V\subseteq |X|$. There exists an equivalence
\begin{align}
\mathcal O_{X,\infty}(V) \simeq \lim_{[n]\in\Delta} \mathcal O_{X,\infty}(V_n)    
\end{align}
in $s\mathbf C^\infty[W_{\mathrm{lax}}^{-1}]$.
According to Proposition \ref{prop:sp-preservation} $\operatorname{sp}=\phi\circ\psi\circ\delta\circ U$ also preserve limits. Therefore,
\begin{align*}
\mathbb O_X^{\mathrm{pre}}(V)
&=
\operatorname{sp}\bigl(\mathcal O_{X,\infty}(V)\bigr)\\
&\simeq
\operatorname{sp}\left(
\lim_{[n]\in\Delta}
\mathcal O_{X,\infty}(V_n)
\right)\\
&\simeq
\lim_{[n]\in\Delta}\operatorname{sp}\bigl(\mathcal O_{X,\infty}(V_n)\bigr)\\
&=\lim_{[n]\in\Delta} \mathbb O_X^{\mathrm{pre}}(V_n).
\end{align*}

Let $\widetilde{\mathbb O}_X$ be the internal spectral realization of the sheaf of simplicial commutative $\mathbb R$-algebras underlying the structure sheaf $\mathcal O_X$. \cite[§3.2.3]{Steffens}. Every open $V$ is paracompact Hausdorff, and $\pi_0\mathcal O_X|_V$ is fine (\cite[Proposition 3.3]{Spivak1} and \cite[Proposition 2.3.89]{Steffens}). Since $\Omega^\infty$ commutes with sections and global sections of the structure module are $t$-exact \cite[Theorem 3.2.49(1)]{Steffens}, we obtain $
\mathbb O_X^{\mathrm{pre}}(V) \simeq\tau_{\ge0}\Gamma(V,\widetilde{\mathbb O}_X) \simeq\Gamma(V,\widetilde{\mathbb O}_X)$ naturally in $V$. Thus the underlying spectral presheaf is a hypersheaf, and its pointwise commutative algebra structure makes it a commutative algebra object in spectral hypersheaves. We now denote $\mathbb O_X:=\mathbb O_X^{\mathrm{pre}}$, there is $\mathbb O_X\in \operatorname{CAlg} \left(\operatorname{Shv}_{\mathrm{hyp}}(|X|,\operatorname{Sp}) \right).$
\end{proof}

\subsection{Spectralized Stalks}
Now we can introduce and consider the spectralized stalk of the spectralized (hyper)sheaf.
\begin{definition}
\emph{
We denote the spectralized stalk of a spectralized (hyper)sheaf as follows}
\begin{align*}
    \mathbb{O}_{X,x} := \operatorname*{colim}_{x \in V} \mathbb{O}_X(V).
\end{align*}
\end{definition}

\begin{proposition}
    \emph{The spectralized stalk $\mathbb{O}_{X,x}$ is a connective local $E_{\infty}$ ring.}
\end{proposition}
\begin{proof}

According to the Proposition~\ref{prop:sp-preservation}, we can easily get $\mathbb O_{X,x}\simeq \operatorname{sp}(\mathcal O_{X,x}).$
It is connective because $sp$ takes values in
$\mathrm{CAlg}_{H\mathbb R}^{\mathrm{cn}}$. Moreover, $\pi_0(\mathbb O_{X,x}) \cong \pi_0\bigl(U(\mathcal O_{X,x})\bigr)$, which is local by \cite[Definitions~5.10 and~6.3]{Spivak1}.
\end{proof}

\section{Module Categories and the Picard Hypersheaf}

\subsection{Modules over the spectralized structure sheaf}

\begin{definition}
\emph{Let $\mathcal X=(|X|,\mathcal O_X)$ be a compact derived smooth
manifold and denote its spectralized structure sheaf is $\mathbb O_X$. For every open subset $V\subseteq |X|$, we define the $\infty$-category of modules over the spectralized structure sheaf as follows
\begin{align*}
\mathcal M_X(V) := \operatorname{Mod}_{\mathbb O_X|_V} \bigl(\operatorname{Shv}_{\mathrm{hyp}}(V,\operatorname{Sp})\bigr).
\end{align*}
}
\end{definition}

\begin{remark}
\emph{As explained in the proof below, $\mathbb O_X$ carries a natural commutative algebra structure in $\operatorname{Shv}_{\mathrm{hyp}}(|X|,\operatorname{Sp})$. Steffens also defines modules over sheaves of simplicial $C^\infty$-rings as modules over the sheaf of its underlying simplicial commutative $\mathbb R$-algebra \cite[Section~3.2.3]{Steffens}. 
Since $\mathbb O_X|_V$ comes from spectralizing the underlying algebra sheaf $\mathcal O_{X,\infty}^{\mathrm{alg}}|_V$, the category
$\mathcal M_X(V)$ is the spectral realization of Steffens'
construction. There exists an equivalence between the category of our module and Steffens' module.}
\end{remark}

\begin{proposition}
\emph{For every open subset $V\subseteq |X|$, the $\infty$-category $\mathcal M_X(V)$ is presentable, stable and symmetric monoidal. Its symmetric monoidal product is given by the relative tensor product $\mathcal F\otimes_{\mathbb O_X|_V}\mathcal G$.}
\end{proposition}

\begin{proof}
For an open subset $V\subseteq |X|$, $\operatorname{Shv}_{\mathrm{hyp}}(V,\operatorname{Sp})$ is presentable, stable and symmetric monoidal, and in every variable its tensor product preserves colimits. Since $\mathbb O_X|_V$ is a sheaf of commutative $H\mathbb R$-algebras it leads to a commutative algebra object $\mathbb O_X|_V\in\operatorname{CAlg}(\operatorname{Shv}_{\mathrm{hyp}}(V,\operatorname{Sp}))$. According to \cite[Corollary~4.2.3.7(1)]{LurieHA}, we know that the $\infty$-category of left $\mathbb O_X|_V$-modules in $\operatorname{Shv}_{\mathrm{hyp}}(V,\operatorname{Sp})$ is presentable. Since $\mathbb O_X|_V$ is commutative, according to \cite[Corollary~4.5.1.6]{LurieHA}, we know that $\mathcal M_X(V)$ is also presentable and we have 
\begin{align*}
    \mathcal M_X(V)=\operatorname{Mod}_{\mathbb O_X|_V}(\operatorname{Shv}_{\mathrm{hyp}}(V,\operatorname{Sp}))\simeq \operatorname{LMod}_{\mathbb O_X|_V}(\operatorname{Shv}_{\mathrm{hyp}}(V,\operatorname{Sp})).
\end{align*}
According to \cite[Theorem~4.5.2.1(1)--(2)]{LurieHA} $\operatorname{Shv}_{\mathrm{hyp}}(V,\operatorname{Sp})$ admits geometric realizations and its tensor product preserves them separately. Consequently,
$\mathcal M_X(V)$ carries a
symmetric monoidal structure whose tensor product is the relative
tensor product
\begin{align*}
(\mathcal F,\mathcal G)
\longmapsto
\mathcal F\otimes_{\mathbb O_X|_V}\mathcal G,
\end{align*}

Finally,we know that the forgetful functor
\begin{align*}
\mathcal M_X(V)=\operatorname{Mod}_{\mathbb O_X|_V}(\operatorname{Shv}_{\mathrm{hyp}}(V,\operatorname{Sp}))
\longrightarrow
 \operatorname{Shv}_{\mathrm{hyp}}(V,\operatorname{Sp})
\end{align*}
is conservative and finite limits and finite colimits of modules are being computed inside $\operatorname{Shv}_{\mathrm{hyp}}(V,\operatorname{Sp})$. Since finite limits and colimits and the initial objects are preserved, spaces and suspensions can be computed in $\operatorname{Shv}_{\mathrm{hyp}}(V,\operatorname{Sp})$. It follows that $\mathcal M_X(V)$ is stable.

\end{proof}

\begin{construction}
\emph{let $j: W \hookrightarrow V$ be an inclusion map, $W$ and $V$ are open subsets in $|X|$, there is a colimit-preserving symmetric monoidal functor}
\begin{align*}
    j^{-1}:  \operatorname{Shv}_{\mathrm{hyp}}(V,\operatorname{Sp}) \longrightarrow 
 \operatorname{Shv}_{\mathrm{hyp}}(W,\operatorname{Sp}).
\end{align*}
\emph{There exists a canonical equivalence between these two commutative objects $j^{-1}(\mathbb{O}_{X}|_{V}) \simeq \mathbb{O}_{X}|_{W}$, this induces the following map $j^{*}_{Mod}: \mathcal{M}_{X}(V) \longrightarrow \mathcal{M}_X(W)$, We assume that $\mathcal{F}, \mathcal{G} \in \mathcal{M}_{X}(V)$, the image of $j^{*}_{Mod}(\mathcal{F})$ is a $\mathbb{O}_{X}|_{W}$-Module, because $j^{-1}(\mathbb{O}_{X}|_{V}) \simeq \mathbb{O}_{X}|_{W}$. Moreover, there exists the canonical equivalence}
\begin{align*}
j^{*}_{Mod}(\mathcal{F} \otimes_{\mathbb{O}_{X}|_{V}}\mathcal{G}) \simeq j^{*}_{Mod}(\mathcal{F}) \otimes_{\mathbb{O}_{X}|_{W}} j^{*}_{Mod}(\mathcal{G}) .
\end{align*}

\emph{These restriction functors are compatible up to canonical coherent equivalence. Finially, all these above lead to the following presheaf:}
\begin{align*}
\mathcal{M}_{X}: \operatorname{Open}(|X|)^{op} &\longrightarrow \operatorname{CAlg}(\operatorname{Pr}^{L
})\\
V & \longmapsto \operatorname{Mod}_{\mathbb O_X|_V}(\operatorname{Shv}_{\mathrm{hyp}}(V,\operatorname{Sp}))
\end{align*}
\end{construction}

\begin{proposition}
\emph{
The presheaf above is a hypersheaf.}
\end{proposition}

\begin{proof}
$V_\bullet\to V$ is a hypercover and $\operatorname{Shv}_{\mathrm{hyp}}(|X|,\mathcal S)$ is hypercomplete, the hypercover is effective, according to \cite[Theorem~6.5.3.12]{LurieHTT}, and there is $V\simeq \operatorname*{colim}_{[n]\in\Delta^{\mathrm{op}}}V_n.$

The construction carries colimits to limits according to \cite[Theorem~6.1.3.9 and Proposition~6.1.3.10]{LurieHTT}. Since stabilization preserves limits \cite[Remark~1.4.2.25]{LurieHA}, this is a symmetric monoidal equivalence
\begin{align*}
    \operatorname{Shv}_{\mathrm{hyp}}(V,\operatorname{Sp}) \simeq \lim_{[n]\in\Delta} \operatorname{Shv}_{\mathrm{hyp}}(V_n,\operatorname{Sp}).
\end{align*}

With this equivalence, $\mathbb O_X|_V$ corresponds to the compatible family $\{\mathbb O_X|_{V_n}\}_{[n]\in\Delta}$. Since the formation of module categories commutes with limits (See \cite[\text{Definition~3.2.28}]{Steffens}, \cite[Corollary~3.2.2.5]{LurieHA}), we can get an following equivalence in $\operatorname{CAlg}\bigl(\operatorname{Pr}^L\bigr)$
\begin{align*}
\mathcal M_X(V)
&= \operatorname{Mod}_{\mathbb O_X|_V} \bigl(\operatorname{Shv}_{\mathrm{hyp}}(V,\operatorname{Sp})\bigr)\\ 
&\simeq\lim_{[n]\in\Delta} \operatorname{Mod}_{\mathbb O_X|_{V_n}} \bigl(\operatorname{Shv}_{\mathrm{hyp}}(V_n,\operatorname{Sp})\bigr)\\
&= \lim_{[n]\in\Delta}\mathcal M_X(V_n).
\end{align*}

\end{proof}

\subsection{Picard Hypersheaf}
We follow the definition of Picard space (See the discussion at the end of \cite[Proposition~4.2]{AntieauGepner}). They claim that, for a presentable symmetric monoidal $\infty$-category $\mathcal C$, the Picard space $\operatorname{Pic}(\mathcal C)$ is the grouplike $E_\infty$-space of tensor-invertible objects and equivalences, so there is a functor
\begin{align*}
\operatorname{Pic}\colon
\operatorname{CAlg}(\operatorname{Pr}^{L})
\longrightarrow
\operatorname{CAlg}^{\mathrm{gp}}(\mathcal S).
\end{align*}

\begin{definition}
\emph{
We now define the Picard Space of $\mathbb{O}_{X}|_{V}$as follows}
\begin{align*}
\operatorname{Pic}_{\mathbb{O}_{X}}(V):= \operatorname{Pic}(\mathcal{M}_{X}(V)) = \operatorname{Pic}(\operatorname{Mod}_{\mathbb O_X|_V}(\operatorname{Shv}_{\mathrm{hyp}}(V,\operatorname{Sp})).
\end{align*}

\end{definition}

\begin{proposition}
\emph{$\operatorname{Pic}_{\mathbb{O}_{X}}$ is a presheaf.}
\end{proposition}
\begin{proof}
Consider the composite of functors, we get $\operatorname{Pic}_{\mathbb{O}_{X}}=\operatorname{Pic}(\mathcal{M}_{X}) = \operatorname{Pic} \circ \mathcal{M}_{X}$. Because $\mathcal{M}_{X}$ is a presheaf, we can automaticlly get that $\operatorname{Pic}_{\mathbb{O}_{X}}$ is a presheaf.
\end{proof}

\begin{proposition}
    \emph{$\operatorname{Pic}_{\mathbb{O}_{X}}$ is a hypersheaf}
\end{proposition}
\begin{proof}
    According to \cite[Proposition 4.3]{AntieauGepner}, we get this immediately.
\end{proof}

\section{Local Structure of the Picard Hypersheaf}

We now start with the local structure of Picard Hypersheaf.

\subsection{Picard space of spectralized stalk}
\begin{definition}
\emph{Recall that $\mathcal X=(|X|,\mathcal O_X)$ is a compact derived smooth manifold and its spectralized structure sheaf is $\mathbb O_X$. For every point $x \in |X|$, we define the local Picard space of the spectralized stalk as follows 
\begin{align*}    
\operatorname{Pic}_{\mathbb O_{X,x}}&:= \operatorname{Pic}(\operatorname{Mod}_{\mathbb O_{X,x}}(\operatorname{Shv}_{\mathrm{hyp}}(*,\operatorname{Sp}))
\\& \simeq  \operatorname{Pic}(\operatorname{Mod}_{\mathbb O_{X,x}}(\operatorname{Sp})).
\end{align*}
}
\end{definition}

\begin{lemma}
\label{map-colimmap}
\emph{Let $x \in V \subseteq |X|$ and $\mathcal{F}, \mathcal{G} \in \mathcal{M}_{X}(V)$ be invertible $\mathbb{O}_{X}|_{V}$-Modules. Then there is a canonical equivalence}
\begin{align*}
    \mathrm{Map}_{Pic_{\mathbb{O}_{X,x}}}(\mathcal F_x,\mathcal{G}_x) \simeq \operatorname*{colim}_{x\in W\subseteq V}\mathrm{Map}_{Pic_{\mathbb{O}_{X}}(W)}(\mathcal{F}_{W},\mathcal{G}_{W}).
\end{align*}

\end{lemma}
\begin{proof}
We follow \cite[Lemma 4.6.1.5.]{LurieHA} by choosing $(\mathcal{C},\mathcal{M},M,N,X,X^{\vee},e,c,\mathbf{1})$ as follows:
\begin{align*}
&\mathcal{C} = \mathcal{M}= \mathcal M_{X}(W), \quad M=\mathcal{G}|_{W},\quad N=\mathbb{O}_{X}|_{W},\\
&X= \mathcal{F}|_{W}, \quad X^{\vee}=(\mathcal{F}|_{W})^{-1},\quad \mathbf{1}= \mathbb O_X|_W\\
&e\colon(\mathcal F|_W)^{-1} \otimes_{\mathbb O_X|_W} \mathcal F|_W\xrightarrow{\simeq} \mathbb O_X|_W,\\
&c\colon \mathbb O_X|_W \xrightarrow{\simeq} \mathcal F|_W \otimes_{\mathbb O_X|_W} (\mathcal F|_W)^{-1}.
\end{align*}
We have an equivalence
\begin{align}
\label{Map-Map} 
\operatorname{Map}_{\mathcal M_X(W)} (\mathcal F|_W,\mathcal G|_W)
&\simeq \operatorname{Map}_{\mathcal M_X(W)} \left(\mathbb O_X|_W, (\mathcal F|_W)^{-1} \otimes_{\mathbb O_X|_W} \mathcal G|_W \right).
\end{align}

By the free–forgetful adjunction for $\mathbb O_X|_W$-modules, we obtain the following equivalence:
\begin{align*}
\operatorname{Map}_{\mathcal M_X(W)}
(\mathbb O_X|_W,(\mathcal F|_W)^{-1} \otimes_{\mathbb O_X|_W} \mathcal G|_W)
\simeq
\Omega^\infty\Gamma(W,(\mathcal F|_W)^{-1} \otimes_{\mathbb O_X|_W} \mathcal G|_W).
\end{align*}
From \ref{Map-Map}, we can see that
\begin{align*}
    \operatorname{Map}_{\mathcal M_X(W)}
(\mathcal F|_W,\mathcal G|_W) \simeq \Omega^\infty\Gamma(W,(\mathcal F|_W)^{-1} \otimes_{\mathbb O_X|_W} \mathcal G|_W).
\end{align*}
We can take the filtered colimit over all neighborhoods $W$ of $x$ contained in $V$ and arrive at
\begin{align*}
\operatorname*{colim}_{x\in W\subseteq V}
\operatorname{Map}_{\mathcal M_X(W)} (\mathcal F|_W,\mathcal G|_W)  &\simeq \operatorname*{colim}_{x\in W\subseteq V} \Omega^\infty\Gamma(W,(\mathcal F|_W)^{-1} \otimes_{\mathbb O_X|_W} \mathcal G|_W)\\
&\simeq
\Omega^\infty
\left(
\operatorname*{colim}_{x\in W\subseteq V}
\Gamma(W,(\mathcal F|_W)^{-1} \otimes_{\mathbb O_X|_W} \mathcal G|_W)
\right).
\end{align*}
Since $\Omega^{\infty}$ preserves filtered colimits and we know that the stalk functor is symmetric monoidal, because of the commutativity between the filtered colimits and the subspaces which consist of equivalences \cite[Proposition 5.3.3.3.]{LurieHTT}, we finally get the stated result
\begin{align*}
\operatorname*{colim}_{x\in W\subseteq V}
\operatorname{Map}_{Pic_{\mathbb{O}_{X}}(W)}
(\mathcal F|_W,\mathcal G|_W)
\simeq
\operatorname{Map}_{\operatorname{Pic}_{\mathbb O_{X,x}}}
(\mathcal F_x,\mathcal G_x).
\end{align*}

\end{proof}

\begin{proposition}
\label{Pic-pic}
\emph{For every neighborhood $V \ni x$, there exists a symmetric monoidal-preserving functor after taking stalks
\begin{align*}
\operatorname{Mod}_{\mathbb O_X|_V} \bigl(\operatorname{Shv}_{\mathrm{hyp}}(V,\operatorname{Sp})\bigr) \longrightarrow \operatorname{Mod}_{\mathbb O_{X,x}}(\operatorname{Sp}).
\end{align*}
This induces $\operatorname{Pic}_{\mathbb{O}_{X}}(V) \longrightarrow \operatorname{Pic}_{\mathbb O_{X,x}}$, and we can take the filtered colimit
\begin{align*}
\theta_x : (\operatorname{Pic}_{\mathbb{O}_{X}} )_{x} = \operatorname*{colim}_{x \in V}\operatorname{Pic}_{\mathbb{O}_{X}}(V) \longrightarrow
\operatorname{Pic}_{\mathbb O_{X,x}}.
\end{align*}
There exists a canonical equivalence between the Picard space of the spectralized stalk $\operatorname{Pic}_{\mathbb O_{X,x}}$ and the stalk of the Picard hypersheaf $(\operatorname{Pic}_{\mathbb{O}_{X}})_{x}$, which means that $\theta_x$ is an equivalence.}
\end{proposition}

\begin{proof}
In order to show that $\theta_x$ is an equivalence, we need to show the following 2 things:

\textbf{1. $\theta_x$ is essentially surjective}

We pick an arbitrary $\mathcal{L} \in Pic_{\mathbb{O}_{X,x}}$. According to \cite[Theorem 7.9.]{AntieauGepner}, there is an equivalence $\mathcal{L} \simeq \Sigma^{n}\mathbb{O}_{X,x}$. Choose an arbitrary $x \in V$ and consider that the module $\Sigma^n \mathbb{O}_{X}|_{V}$ is invertible in $\mathcal{M}_{X}(V)$, so it actually defines an object in the stalk $[\Sigma^{n}\mathbb{O}_{X}|_{V}]_{x} \in (Pic_{\mathbb{O}_{X}})_{x}$. Now consider the image of $\theta_x$:
\begin{align*}
    \theta_{x}([\Sigma^{n}\mathbb{O}_{X}|_{V}]_{x}) 
    &\simeq (\Sigma^{n}\mathbb{O}_{X}|_{V})_{x} \\
    &\simeq \Sigma^{n}(\mathbb{O}_{X}|_{V})_{x} \\
    &\simeq \Sigma^{n}\mathbb{O}_{X,x}\\
    &\simeq \mathcal{L}.
\end{align*}

\textbf{2. $\theta_x$ is fully faithful:}
We take arbitrary two elements from $[\mathcal{F}]_{x}, [\mathcal{G}]_{x} \in (\operatorname{Pic}_{\mathbb O_{X}})_{x}$ that can be represented by two invertible $\mathcal{F},\mathcal{G} \in \mathcal{M}_{X}(V)$. According to \ref{map-colimmap}, the following canonical map is an equivalence
\begin{align*}
\operatorname{Map}_{(\operatorname{Pic}_{\mathbb O_X})_x}
\bigl([\mathcal F]_x,[\mathcal G]_x\bigr)
\longrightarrow
\operatorname{Map}_{\operatorname{Pic}_{\mathbb O_{X,x}}}
\bigl(\mathcal F_x,\mathcal G_x\bigr).
\end{align*}
Therefore, $\theta_x$ is fully faithful.

\end{proof}

\subsection{Local Homotopy Type of the Picard Hypersheaf}

\begin{proposition}
\emph{We have an equivalence
\begin{align*}
(\operatorname{Pic}_{\mathbb O_X})_x \simeq \operatorname{Pic}_{\mathbb O_{X,x}} \simeq \mathbb Z\times BGL_1(\mathbb O_{X,x}).
\end{align*}}
\end{proposition}

\begin{proof}
We know $(\operatorname{Pic}_{\mathbb O_X})_x \simeq \operatorname{Pic}_{\mathbb O_{X,x}}$ from Proposition \ref{Pic-pic}. According to Theorem 7.9. and Corollary 7.10. from \cite{AntieauGepner}, it follows immediately that $\operatorname{Pic}_{\mathbb O_{X,x}} \simeq \mathbb Z\times BGL_1(\mathbb O_{X,x})$.
\end{proof}

\section{Picard Twists and Their Cohomological Classification}
\subsection{Global Splitting of the Picard Hypersheaf}

Now we can define the presheaf of units.
\begin{definition}
\emph{Recall that $\mathcal X=(|X|,\mathcal O_X)$ is a compact derived smooth manifold and its spectralized structure sheaf is $\mathbb O_X$. For every open set $V \subseteq |X|$, we can define the hypersheaf of units as follows
\begin{align*}
GL_1(\mathbb O_X)(V) := GL_1(\mathbb O_X(V))
\end{align*}
which means that $GL_1(\mathbb O_X(V))$ can be defined as the following pullback square:
}
\begin{center}
\begin{tikzcd}[column sep=large, row sep=large]
GL_1(\mathbb O_X)(V)
  \arrow[r]
  \arrow[d]
&
\Omega^{\infty}\mathbb O_X(V)
  \arrow[d]
\\
(\pi_{0}\mathbb O_X(V))^{\times}
  \arrow[r]
&
\pi_{0}\mathbb O_X(V)
\end{tikzcd}
\end{center}
\emph{Since the restriction map preserves units we can conclude that this induces a presheaf}
\begin{align*}
GL_1(\mathbb O_X)\colon
\operatorname{Open}(|X|)^{\mathrm{op}}
\longrightarrow
\operatorname{CAlg}^{\mathrm{gp}}(\mathcal S).
\end{align*}
\end{definition}
\begin{remark}
\emph{Both $GL_1(\mathbb O_X)$ and $\operatorname{Pic}$ can be written as}
$\operatorname{Open}(|X|)^{\mathrm{op}} \longrightarrow \operatorname{CAlg}^{\mathrm{gp}}(\mathcal S),$
\emph{ but they arise by different compositions.}
\begin{align*}
GL_1(\mathbb O_X) &:\operatorname{Open}(|X|)^{\mathrm{op}}
\xrightarrow{\mathbb O_X}
\operatorname{CAlg}_{H\mathbb R}^{\mathrm{cn}}
\xrightarrow{U_{\mathrm{alg}}}
\operatorname{CAlg}(\operatorname{Sp})^{\mathrm{cn}}
\xrightarrow{GL_1}
\operatorname{CAlg}^{\mathrm{gp}}(\mathcal S)
\\
\operatorname{Pic}_{\mathbb O_X}&: \operatorname{Open}(|X|)^{\mathrm{op}} \xrightarrow{\mathcal{M}_{X}}  \operatorname{CAlg}(\operatorname{Pr}^{L}) \xrightarrow{\operatorname{Pic}} \operatorname{CAlg}^{\mathrm{gp}}(\mathcal S)
\end{align*}
\end{remark}

\begin{proposition}
\emph{The presheaf $GL_1(\mathbb{O}_X)$ is a hypersheaf.}
\end{proposition}

\begin{proof}
Let  $V_{\bullet} \longrightarrow V$ be a hypercover of an open set $V \subseteq |X|$. Since the spectralized structure sheaf $\mathbb{O}_{X}$ is a hypersheaf, by Proposition \ref{O_X-hypersheaf} there exists a canonical equivalence $\mathbb O_X(V) \simeq \lim_{[n]\in\Delta} \mathbb O_X(V_n)$ in $\operatorname{CAlg}_{H\mathbb R}^{\mathrm{cn}}$. The forgetful functor $U_\mathrm{alg}$ and $GL_1$ preserve limits, so that we have
\begin{align*}
GL_1(\mathbb O_X)(V)
&=
GL_1\bigl(U_{\mathrm{alg}}(\mathbb O_X(V))\bigr)\\
&\simeq
GL_1\left(
\lim_{[n]\in\Delta}
U_{\mathrm{alg}}\bigl(\mathbb O_X(V_n)\bigr)
\right)\\
&\simeq
\lim_{[n]\in\Delta}
GL_1\bigl(U_{\mathrm{alg}}(\mathbb O_X(V_n))\bigr)\\
&=
\lim_{[n]\in\Delta}
GL_1(\mathbb O_X)(V_n).
\end{align*}
Thus, the presheaf $GL_1(\mathbb{O}_X)$ is a hypersheaf

\end{proof}

\begin{proposition}
\emph{For every $x \in |X|$, there exists a canonical equivalence}
\begin{align*}
\bigl(GL_1(\mathbb O_X)\bigr)_x \simeq GL_1(\mathbb O_{X,x}).
\end{align*}
\end{proposition}

\begin{proof}
Since $\Omega^{\infty}$ preserves filtered colimits (See Lemma \ref{map-colimmap},), and filtered colimits of spaces commute with finite limits by \cite[Proposition 5.3.3.3]{LurieHTT}, also $\pi_0$ and formation of units in commutative rings preserves filtered colimits, so there is
\begin{align*}
\bigl(GL_1(\mathbb O_X)\bigr)_x
&=
\operatorname*{colim}_{x\in V}
GL_1(\mathbb O_X(V))\\
&\simeq
GL_1\left(
\operatorname*{colim}_{x\in V}\mathbb O_X(V)
\right)\\
&\simeq
GL_1(\mathbb O_{X,x}).
\end{align*}
\end{proof}

\begin{proposition}
\emph{Let $\mathrm{B}GL_1(\mathbb{O}_X)$ be the classifying hypersheaf of $GL_1(\mathbb{O}_X)$, for every $x \in |X|$, there exists a canonical equivalence}
\begin{align*}
\bigl(\mathrm{B}GL_1(\mathbb O_X)\bigr)_x\simeq \mathrm{B}GL_1(\mathbb O_{X,x}).
\end{align*}
\end{proposition}
\begin{proof}
The stalk functor preserves finite limits and colimits, it commutes with the formation of classifying objects. This yields
\begin{align*}
\bigl(\mathrm{B}GL_1(\mathbb O_X)\bigr)_x \simeq \mathrm{B}\bigl((GL_1(\mathbb O_X))_x\bigr) \simeq \mathrm{B}GL_1(\mathbb O_{X,x}).
\end{align*}
\end{proof}

\begin{notation}
\emph{We denote $\underline{\mathbb{Z}}$ to be the hypersheaf associated to the constant presheaf with value $\mathbb Z$}.
\end{notation}

\begin{proposition}
\label{Global:Pic=Z*BGL}
\emph{There exists a natural split fiber sequence of hypersheaves of grouplike $\mathbb{E}_1-$spaces}
\begin{align*}
\mathrm{B}GL_1(\mathbb{O}_X) \longrightarrow \operatorname{Pic}_{\mathbb O_X} \longrightarrow \underline{\mathbb Z}.
\end{align*}
\emph{This yields an equivalence}
\begin{align*}
\operatorname{Pic}_{\mathbb O_X} \simeq \underline{\mathbb Z} \times \mathrm{B} GL_1(\mathbb{O}_X).
\end{align*}
\end{proposition}

\begin{proof}
We can first take the stalk, then consider the following diagram:
\begin{center}
\begin{tikzcd}[column sep=large, row sep=large]
(\mathrm{B} GL_1 (\mathbb O_{X}))_{x}
  \arrow[r]
  \arrow[d,"\simeq"]
&
(\operatorname{Pic}_{\mathbb O_{X}})_x
  \arrow[d,"\simeq"]
  \arrow[r]
&
(\underline{\mathbb Z})_{x}
  \arrow[d,"\simeq"]
\\
\mathrm{B} GL_1 (\mathbb O_{X,x})
  \arrow[r]
&
\operatorname{Pic}_{\mathbb O_{X,x}}
  \arrow[r]
&
\mathbb Z
\end{tikzcd}
\end{center}
The first map is induced by automorphisms of the tensor unit, and the second records the locally constant suspension degree of an invertible module. The three vertical arrows in the diagram represent proven canonical equivalences. The bottom row is the split fiber sequence of spaces given by \cite[Corollary 7.10]{AntieauGepner}. Taking stalk preserves finite limits and the morphism of hypersheaves is an equivalence if and only if it induces an equivalence on every stalk. Thus, the following global sequence $
\mathrm{B}GL_1(\mathbb{O}_X) \longrightarrow \operatorname{Pic}_{\mathbb O_X} \longrightarrow \underline{\mathbb Z}$
is also a fiber sequence. There exists a suspension map
\begin{align*}
\underline{\mathbb Z} \longrightarrow \operatorname{Pic}_{\mathbb O_X}, \qquad n \longmapsto \Sigma^n \mathbb O_X,
\end{align*}
which gives a section of the second map, so this is a split sequence. Naturally, we arrive at
\begin{align*}
\operatorname{Pic}_{\mathbb O_X} \simeq \underline{\mathbb Z} \times \mathrm{B} GL_1(\mathbb{O}_X).
\end{align*}
\end{proof}

\begin{remark}
\emph{According to \cite[Lemma A.3.9]{LurieHA}, a morphism of sheaves on the topological space $|X|$ induces an equivalence on every stalk if and only if it is $\infty$-connective. Since $\infty$-connective morphisms between hypercomplete objects are equivalences, equivalences of hypersheaves can be checked stalkwise.}
\end{remark}

\subsection{Picard Twists}
\begin{definition}
\emph{Let $\mathrm{B} \operatorname{Pic}_{\mathbb O_X}$ be the classifying hypersheaf of $\operatorname{Pic}_{\mathbb O_X}$. We call the following global section $\alpha \in \Gamma (|X|,\mathrm{B} \operatorname{Pic}_{\mathbb O_X})$ a $\mathbb{O}_X$-Picard twist. The equivalence class of $\mathbb{O}_X$-Picard twists forms a set $\pi_0 \Gamma (|X|,\mathrm{B} \operatorname{Pic}_{\mathbb O_X})$}.
\end{definition}

\begin{proposition}
\emph{
There exists a natural equivalence}
\begin{align*}
    \mathrm{B}\operatorname{Pic}_{\mathbb O_X} \simeq \mathrm{B}\underline{\mathbb Z} \times \mathrm{B}^2 GL_1(\mathbb{O}_X).
\end{align*}
\end{proposition}
\begin{proof}
We take a classifying object in \ref{Global:Pic=Z*BGL} and apply the property that this construction preserves finite products. This proves the proposition.
\end{proof}

\begin{corollary}
\emph{For the equivalence class of $\mathbb{O}_X$-Picard twists, there exists a natural bijection}
\begin{align*}
\pi_0\Gamma\bigl(|X|,\mathrm{B}\operatorname{Pic}_{\mathbb O_X}\bigr)
\cong
H^1(|X|;\underline{\mathbb Z})
\times
\pi_0\Gamma\bigl(|X|,\mathrm{B}^2 GL_1(\mathbb{O}_X)\bigl).
\end{align*}
\end{corollary}

We now construct the twisted module categories.

\begin{construction}
\emph{For a fixed Picard twist }$\alpha \in \Gamma \bigl ( |X|, \mathrm{B} \operatorname{Pic}_{\mathbb O_X} \bigl)$\emph{, we choose an open cover $\{ U_i\}$ and a trivialization of $\alpha$ over $U_i$. On each overlap of $U_{ij} = U_i \cap U_j$, the change between the chosen trivialization of $\alpha$ is represented by an invertible module $\mathcal{L}_{ij}  \in \operatorname{Pic}_{\mathbb{O}_X}(U_{ij})$ and the equivalence of triple intersections}
\begin{align*}
\mu_{ijk} : \mathcal L_{ij}|_{U_{ijk}} \otimes_{\mathbb O_X|_{U_{ijk}}} \mathcal L_{jk}|_{U_{ijk}} \xrightarrow{\simeq} \mathcal L_{ik}|_{U_{ijk}}.
\end{align*}
\emph{The transition data between the trivializations also determine 
the unit data and all higher coherence homotopies. Every $\mathcal{L}_{ij}$ defines an autoequivalence of $\mathcal{M}_X(U_{ij})$}:
\begin{align*}
   \mathcal{L}_{ij}\otimes -: \mathcal{M}_X(U_{ij}) &\longrightarrow \mathcal{M}_X(U_{ij})\\ \mathcal{F} &\longmapsto \mathcal{L}_{ij}\otimes_{\mathbb O_X|_{U_{ij}}}\mathcal F 
\end{align*}
\emph{Its inverse is $\mathcal{L}_{ij}^{-1} \otimes -$ and $\mu_{ijk}$ keeps the compatibility of the composite of these functors.}
\end{construction}

\begin{remark}
\emph{$\operatorname{Pic}_{\mathbb O_X}$ is a grouplike $\mathbb{E}_1$-object, the basepoint morphism $*\longrightarrow B\operatorname{Pic}_{\mathbb O_X}$ is an effective epimorphism \cite[Example 5.2.6.13]{LurieHA}, In other words, every section of $\mathrm{B}\operatorname{Pic}_{\mathbb O_X}$ is locally equivalent to the basepoint. Hence, there exists an open cover $\{U_i\}$ of $|X|$ such that each restriction $\alpha|_{U_i}$ is trivial.}
\end{remark}

We are now going to define $\alpha$-twisted $\infty-$ module categories.

\begin{definition}
\emph{Let $\alpha \in \Gamma(|X|, \mathrm{B} \operatorname{Pic}_{\mathbb{O}_{X}})$ be a $\mathbb{O}_X$-Picard twist, we denote $\mathcal{P}_{\alpha}$ to be a $\operatorname{Pic}_{\mathbb O_X}$-torsor, which is classified by $\alpha$. The relative tensor provides the following group action:}
\begin{align*}
    \operatorname{Pic}_{\mathbb O_X} \times \mathcal{M}_{X} &\longrightarrow \mathcal{M}_{X}\\
    (\mathcal{L},\mathcal{F}) &\longmapsto \mathcal{L}\otimes_{\mathbb O_X}\mathcal{F}
\end{align*}
\emph{We define the hypersheaf of $\alpha$-twisted $\mathbb O_X$-modules as this $\mathcal M_X^\alpha := \mathcal P_\alpha \times^{\operatorname{Pic}_{\mathbb O_X}} \mathcal M_X \in \operatorname{Shv}_{\mathrm{hyp}} \bigl(|X|,\operatorname{Pr}^{L}\bigr)$, so for an open subset $V \subset |X|$, we define the $\alpha$-twisted module $\infty$-category as follows}
\begin{align*}
    \mathcal{M}_X^{\alpha}(V) := \Gamma(V, \mathcal M_X^\alpha ).
\end{align*}
\emph{So the global $\alpha$-twisted module $\infty$-category is $
    \operatorname{Mod}_{\mathbb O_X}^{\alpha} := \mathcal{M}_X^{\alpha}(|X|) = \Gamma(|X|,\mathcal M_X^\alpha ).$
}
\end{definition}

\begin{proposition}
\emph{Let $U \subseteq |X|$ be an open subset and $\alpha|_{U}$ be trivial. Every choice of a trivialization of $\alpha|_{U}$ induces an equivalence $\mathcal M_X^\alpha|_U \simeq \mathcal M_X|_U $ in $
\operatorname{Shv}_{\mathrm{hyp}}
\bigl(U,\operatorname{Pr}^L\bigr)$}.
\end{proposition}

\begin{proof}
A trivialization of $\alpha|_U$ will identify the torsor $\mathcal P_\alpha|_U$ as a trivial torsor, so we have
\begin{align*}
\mathcal M_X^\alpha|_U &= \bigl( \mathcal P_\alpha \times^{\operatorname{Pic}_{\mathbb O_X}} \mathcal M_X \bigl)\bigg|_U\\
&\simeq\operatorname{Pic}_{\mathbb O_X}|_U\times^{\operatorname{Pic}_{\mathbb O_X}|_U} \mathcal M_X|_U\\
&\simeq \mathcal M_X|_U.
\end{align*}
\end{proof}

\begin{remark}
\emph{
For a chosen open cover $\{U_{i}\}$ there exists an equivalence $
\mathcal M_X^\alpha|_{U_i} \simeq \mathcal M_X|_{U_i}$. On the intersection $U_{ij}$, the transition of equivalences between two trivializations is given by
\begin{align*}
\mathcal L_{ij}
\otimes_{\mathbb O_X|_{U_{ij}}}-:
\mathcal M_X|_{U_{ij}}
\longrightarrow
\mathcal M_X|_{U_{ij}}
\end{align*}}.
\end{remark}

\begin{corollary}
\emph{
For every open subset $V \subseteq |X|$, the categories $\mathcal{M}_{X}^\alpha(V)$ and $\operatorname{Mod}_{\mathbb O_X}^{\alpha} = \Gamma(|X|,\mathcal M_X^\alpha)$ are presentable and stable $\infty$-categories.}
\end{corollary}

\begin{proof}
Presentability comes from $\operatorname{Pr}^L$, and we choose a trivial open cover $\{U_{i}\}$ for $\alpha$. Then, we have by descent
\begin{align*}
\mathcal M_X^\alpha(V)
\simeq
\operatorname*{lim}_{[n]\in\Delta}
\prod_{i_0,\ldots,i_n}
\mathcal M_X^\alpha
\bigl(V\cap U_{i_0}\cap\cdots\cap U_{i_n}\bigr).
\end{align*}
By local triviality, the categories in this diagram are stable and their transition functors are exact. So the limit is stable by \cite[Theorem 1.1.4.4]{LurieHA}.
\end{proof}

\section{Gluing Local Compact Generators into Global compact Generators}
In the last section we have constructed the  Picard hypersheaf and its classifying module category $\mathcal{M}_X^\alpha$. We already proved that it is locally equivalent to $\mathcal{M}_X$. We are going to use the compactness and finite closed cover of $|X|$ to study how to construct the global campact generators from local generators. We will then use the internal endomorphism algebra of the generator to construct a derived Azumaya algebra representing the corresponding Picard twist.

\subsection{Closed Covers and Rings of Germs}

\begin{lemma}
\emph{
Let $\{U_i\}_{i\in I}$ be an open cover on which $\alpha$ is trivialized, then there exist finite closed compact subsets $K_1,...,K_n \subseteq |X|$, so that}
\begin{align*}
    K_{a} \subseteq U_{i(a)}, \qquad |X|=\bigcup_{a=1}^{n} \operatorname{int}_{|X|}(K_a).
\end{align*}
\end{lemma}

\begin{proof}
$|X|$ is a compact Hausdorff space. We first choose a finite subcover and then apply the shrinking lemma \cite[proof of Theorem 36.1, Step 1]{Munkres} to obtain an open cover $\{W_a\}_{a=1}^{n}$ satisfying $\overline{W_a}\subseteq U_{i(a)}$. Setting $K_a=\overline{W_a}$ gives the desired conclusion.
\end{proof}

\begin{definition}
\emph{Let $K \subseteq |X|$ be a closed compact subset and $U$ be an open subset. We define the ring of germs of the spectralized structure sheaf along $K$ as follows:}
\begin{align*}
    \mathbb O_{X}(K) := \operatorname*{colim}_{K \subseteq U \subseteq |X| } \mathbb O_X (U) \in CAlg_{H \mathbb R}^{cn}
\end{align*}
\emph{Here, the colimit ranges over all open neighborhoods $U$ of $K$. The transition maps are given by restriction.}
\end{definition}

\begin{definition}
\emph{Let $K \subseteq |X|$ be a closed subset. We denote $i_{K}: K \hookrightarrow |X|$ as an inclusion map. We define the spectralized structure sheaf over $K$ as $\mathbb O_X|_K := i^{-1}_{K} \mathbb O_X$}.
\end{definition}

\begin{remark}
    \emph{Here, $K$ takes the subspace topology and $i_K^{-1}$ means the inverse-image functor in hypersheaf categories.}
\end{remark}

\begin{definition}
\emph{Let $K \subseteq |X|$ be a closed subset. We define the module category of the closed subset $K$ as follows:}
\begin{align*}
\mathcal M_X(K) := \operatorname{Mod}_{\mathbb O_X|_K} \bigl(\operatorname{Shv}_{\mathrm{hyp}}(K,\operatorname{Sp})\bigr).
\end{align*}
\end{definition}

\begin{remark}
\label{hypercomplete_of_Shv_(K,S)}
\emph{From \cite[Proposition 3.3.]{Spivak1}, we know that $|X|$ can be  embedded into some Euclidean space $\mathbb{R}^N$. So, every closed subset $K \subseteq |X|$ is a compact Hausdroff space, which has a covering dimension of $N$ or less. From \cite[Theorem 7.2.3.6 and Corollary 7.2.1.12]{LurieHTT}, we know that $\operatorname{Shv}(K,\mathcal{S})$ is already hypercomplete. The category $\operatorname{Shv}_{\mathrm{hyp}}(K,\operatorname{Sp})$ is a stable, presentable and symmetric monoidal $\infty$-category, and its tensor product preserves colimits in each variable. By the same argument for module categories as in the open case, we know that $\mathcal M_X(K)$ is also stable and presentable.}
\end{remark}

\begin{proposition}
\emph{Let $K \subseteq |X|$ be an arbitrary closed subset. There exists a canonical equivalence $\mathbb O_X(K) \simeq \Gamma(K,\mathbb O_X|_K)$}.
\end{proposition}

\begin{proof}
The global sections of the restriction sheaf over the compact subsets can be calculated by their open neighborhood's sections \cite[Section 7.3.4]{LurieHTT}. Also, the category $\operatorname{CAlg}_{H\mathbb R}^{\mathrm{cn}}$ is presentable, and its filtered colimits are left exact. From  \cite[Theorem 7.3.4.9]{LurieHTT} there exists a canonical equivalence
\begin{align*}
    \Gamma(K,\mathbb O_X|_K) \simeq \operatorname*{colim}_{K \subseteq U \subseteq |X|} \mathbb{O}_X(U) = \mathbb O_X(K).
\end{align*} 
\end{proof}

\subsection{Compact Generators on Closed Subsets}

\begin{proposition}
\emph{Let $K \subseteq |X|$ be an arbitrary closed subset. The spectralized structure sheaf $\mathbb O_X|_K$ is a compact object in $\mathcal{M}_X(K)$}.
\end{proposition}

\begin{proof}
Since $|X|$ is a compact Hausdorff space, so is the closed subset $K$. The global sections functor on sheaves of spaces over $K$ preserves filtered colimits \cite[Corollary 7.3.4.12]{LurieHTT}. Since filtered colimits of spectrum objects are computed levelwise, the global sections functor also preserves filtered colimits of sheaves of spectra, see Remark \ref{hypercomplete_of_Shv_(K,S)}. For an arbitrary $\mathcal{F} \in \mathcal{M}_X(K)$, there exists a canonical equivalence
\begin{align*}
\operatorname{Map}_{\mathcal M_X(K)} (\mathbb O_X|_K,\mathcal F) \simeq \Omega^\infty\Gamma(K,\mathcal F).
\end{align*}
Let $\{\mathcal F_i\}_{i\in I}$ be a filtered diagram in $\mathcal M_X(K)$. The forgetful functor of modules preserves filtered colimits, and $\Omega^\infty$ preserves filtered colimits of spectra. Therefore

\begin{align*}
\operatorname{Map}_{\mathcal M_X(K)}
\left(\mathbb O_X|_K,\operatorname*{colim}_{i\in I}\mathcal F_i\right)
&\simeq
\Omega^\infty\Gamma\left(K,\operatorname*{colim}_{i\in I}\mathcal F_i\right)\\
&\simeq
\operatorname*{colim}_{i\in I}
\Omega^\infty\Gamma(K,\mathcal F_i)\\
&\simeq
\operatorname*{colim}_{i\in I}
\operatorname{Map}_{\mathcal M_X(K)}
(\mathbb O_X|_K,\mathcal F_i).
\end{align*}
So $\mathbb O_X|_K$ is a compact object in $\mathcal{M}_{X}(K)$.
\end{proof}

\begin{proposition}
\label{O_X|_L-is-generator}
\emph{Let $K \subseteq |X|$ be an arbitrary closed subset. The spectralized structure sheaf $\mathbb O_X|_K$ is a generator of $\mathcal{M}_X(K)$}.
\end{proposition}

\begin{proof}
To prove that $\mathbb{O}_X|_K$ is a generator of $\mathcal{M}_{X}(K)$, we need to prove the following:

Given an arbitrary $\mathcal{F}\in\mathcal{M}_X(K)$, 
\begin{align*}
\text{For all } n\in\mathbb{Z}, \text{there is }\operatorname{Map}_{\mathcal{M}_X(K)} \bigl(\Sigma^n(\mathbb{O}_X|_K),\mathcal{F}\bigr)\simeq * \quad \Longrightarrow \quad\mathcal{F}\simeq0.
\end{align*}
We have already considered the map formula of unit modules. Therefore, it is equivalent to prove that:

Given an arbitrary $\mathcal{F}\in\mathcal{M}_X(K)$, $
\Gamma(K,\mathcal{F})\simeq0 \Longrightarrow \mathcal{F} \simeq 0.$

The ringed space $(|X|,\pi_0\mathcal{O}_X)$ is a $C^\infty-$scheme and $|X|$ is Lindel\"of and Hausdorff, so the sheaf $\pi_0\mathcal{O}_X$ is fine and admits partitions of unity for arbitary covers \cite[Proposition 2.3.89]{Steffens}. The spectralization keeps the underlying $\pi_0$-sheaf, so $\pi_0\mathbb{O}_X$ is a fine sheaf and admits partitions of unity for arbitary covers. We now take the open cover of $K$ as $\{K\cap U_i\}$, where each $U_i$ is open in $|X|$. $K$ is closed, so $\{U_i\}\cup\{|X|\setminus K\}$ is an open cover of $|X|$. Choose a partition of unity on $|X|$ which is subordinate to
$\{U_i\}\cup\{|X|\setminus K\}$ and restrict each of its functions to $K$. The term corresponding to $|X|\setminus K$ restricts to zero. The remaining terms form a partition of unity on $K$
which is subordinate to the induced open cover $\{K\cap U_i\}$.

The global sections functor is t-exact \cite[Theorem 3.2.49 (1)]{Steffens}, since $K$ is compact Hausdorff and $\pi_0(\mathbb{O}_X|_K)$ is fine. Then, we have $\pi_n\Gamma(K,\mathcal{F}) \cong \Gamma(K,\pi_n\mathcal{F})$. If $\Gamma(K,\mathcal{F})\simeq0$, since $t$-exactness, there is
\begin{align*}
    \Gamma(K,\pi_n\mathcal{F})=0 \quad \text{for all } n\in\mathbb{Z}.
\end{align*}
By using the partition of unity, every section germ of every homotopy sheaf can be shown by a global section, so all the homotopy sheaves are zero. By the hypercompleteness of $\mathcal{F}$, we have $\mathcal{F} \simeq 0$.

\end{proof}

We now know that $\mathbb{O}_X|_K$ is a compact generator. We can now define the twisted module category of closed subsets.

\begin{definition}
\emph{Let $K \subseteq |X|$ be a closed subset and $W \subseteq K$ be an open subset of $K$. We denote the inclusion map $i_K: K \hookrightarrow |X|$ and now define the hypersheaf $\mathcal M_{X,K}$ of module categories over $K$ for every open subset $W$ as follows:}
\begin{align*}
\mathcal M_{X,K}(W) := \Gamma(W,\mathcal{M}_{X,K}) = \operatorname{Mod}_{(\mathbb O_X|_K)|_W} \bigl(\operatorname{Shv}_{\mathrm{hyp}}(W,\operatorname{Sp})\bigr).
\end{align*}

\emph{The pullback $i_K^{-1}\mathcal P_\alpha$ is a torsor under $i_K^{-1}\operatorname{Pic}_{\mathbb O_X}$. The group-object acts on $\mathcal M_{X,K}$ by restriction of invertible modules and the relative tensor product. Now we can define the $\infty$-category of $\alpha$-twisted $\mathbb O_X|_K$-modules over $K$ as follows:}
\begin{align*}
\mathcal M_X^\alpha(K) := \Gamma\left( K,\, (i_K^{-1}\mathcal P_\alpha) \times^{\,i_K^{-1}\operatorname{Pic}_{\mathbb O_X}} \mathcal M_{X,K} \right).
\end{align*}
\end{definition}

\begin{proposition}
\label{MaX=MX}
\emph{Let $K \subseteq U \subseteq |X|$, where $K$ is closed and $U$ is open. Then, a trivialization of $\alpha|_U$ induces the equivalence $\mathcal M_X^\alpha(K)\simeq\mathcal M_X(K)$ and for $\mathcal M_X^\alpha(K)$ exists a compact generator under the equivalence. The compact generator of $\mathcal M_X^\alpha(K)$ corresponds to $\mathbb O_X|_K$}.
\end{proposition}

\begin{proof}
Since $K\subseteq U$, we pullback the trivialization of $\alpha|_U$ to $K$ and get $i_K^{-1}\mathcal P_\alpha \simeq i_K^{-1}\operatorname{Pic}_{\mathbb O_X}$. Then, there is
\begin{align*}
\mathcal M_X^\alpha(K)
&\simeq
\Gamma\left(
K,\,
(i_K^{-1}\operatorname{Pic}_{\mathbb O_X})
\times^{\,i_K^{-1}\operatorname{Pic}_{\mathbb O_X}}
\mathcal M_{X,K}
\right)\\
&\simeq \Gamma(K,\mathcal M_{X,K})\\
&=\mathcal M_X(K).
\end{align*}

We already know that $\mathbb O_X|_K$ is the compact generator of $\mathcal M_X(K)$. Since equivalences between categories preserve compactness and the property of being a generator, its preimage under the above equivalence is a compact generator of $\mathcal M_X^\alpha(K)$.
\end{proof}

\subsection{Descent for Closed Covers}

In this section we use the notation $K = K_1 \cup K_2$ and $K_{12} = K_1 \cap K_2$.
\begin{proposition}
\emph{Let $K_1$ and $K_2$ be closed subsets in $|X|$. The restriction functors induce the following canonical equivalence:}
\begin{align*}
\operatorname{Shv}_{\mathrm{hyp}}(K,\operatorname{Sp})
\simeq
\operatorname{Shv}_{\mathrm{hyp}}(K_1,\operatorname{Sp})
\times_{\operatorname{Shv}_{\mathrm{hyp}}(K_{12},\operatorname{Sp})}
\operatorname{Shv}_{\mathrm{hyp}}(K_2,\operatorname{Sp}).
\end{align*}
\end{proposition}

\begin{proof}
There are inclusion maps 
\begin{align*}
i_a:K_a\hookrightarrow K,\quad j_a:K_{12}\hookrightarrow K_a \quad(a=1,2), \quad i_{12}:K_{12}\hookrightarrow K.
\end{align*}
For arbitrary closed embeddings $i_a: K_a \hookrightarrow K$ and $j_a: K_{12} \hookrightarrow K_a$, the restriction and pushforward forms adjunctions as follows
\begin{align*}
i^{-1}_a: \operatorname{Shv}_{\mathrm{hyp}}(K,\operatorname{Sp}) &\rightleftarrows \operatorname{Shv}_{\mathrm{hyp}}(K_a,\operatorname{Sp}) :i_{a*}\\
j^{-1}_a: \operatorname{Shv}_{\mathrm{hyp}}(K
_a,\operatorname{Sp}) &\rightleftarrows \operatorname{Shv}_{\mathrm{hyp}}(K_{12},\operatorname{Sp}) :j_{a*}.
\end{align*}

See more results about the correspondce between closed embeddings and closed subtopoi in \cite[Corollary 7.3.2.10]{LurieHTT}. We take an arbitrary $\mathcal{F}_{1} \in \operatorname{Shv}_{\operatorname{hyp}}(K_1,\operatorname{Sp})$ and for every open subset $W \subseteq K$ the pushforward is given by $(i_{1*}\mathcal F_1)(W)=\mathcal F_1(W\cap K_1)$, so that we have the following:

If $x \in K_1$, we can pushforward the sheaf of spectra over $K_1$ to $K$ and the stalk over $x$ doesn't change. Since we choose subspace topology for $K_1$, we can write
\begin{align*}
(i_{1*}\mathcal F_1)_x
\simeq \operatorname*{colim}_{x\in W\subseteq K} \mathcal F_1(W\cap K_1) \simeq(\mathcal F_1)_x.
\end{align*}

If $x \notin K_1$, because $K_1$ is a closed subset, we can always find an open subset without $K_1$, so at this time $(i_{1*}\mathcal F_1)_x = 0$. The same holds for $K_2$.

We now consider the following restriction functor 
\begin{align*}
\mathfrak{R}:\operatorname{Shv}_{\mathrm{hyp}}(K,\operatorname{Sp}) \longrightarrow \operatorname{Shv}_{\mathrm{hyp}}(K_1,\operatorname{Sp}) \times_{\operatorname{Shv}_{\mathrm{hyp}}(K_{12},\operatorname{Sp})} \operatorname{Shv}_{\mathrm{hyp}}(K_2,\operatorname{Sp}).
\end{align*}

The pullback category on the right side of the functor $\mathfrak R$ is made of the triple $(\mathcal F_1, \mathcal F_2, \theta)$, with $\mathcal{F}_a \in \operatorname{Shv}_{\operatorname{hyp}}(K_a,\operatorname{Sp})$ and $\theta: j_1^{-1}\mathcal{F}_1 \xrightarrow{\simeq} j_2^{-1} \mathcal{F}_2$. We denote $\mathcal F_{12}:=j_1^{-1}\mathcal F_1$ and $\eta_a:\mathrm{id}\to j_{a*}j_a^{-1}$ as the unit of the adjunction $j_a^{-1}\dashv j_{a*}$. The first arrow in the diagram can be written as $i_{a*}(\eta_{a})$, for $a=1,2$. There is $i_a\circ j_a=i_{12}$, so that we have $i_{a*}j_{a*}\simeq i_{12*}$. The diagonal arrow is induced by $i_{12*}(\theta^{-1})$. Then we can have the following diagram:
\begin{center}
\begin{tikzcd}[column sep=normal, row sep=normal]
i_{1*}\mathcal{F}_1 
  \arrow[r,"i_{1*}(\eta_{1})"]
&
i_{1*}(j_{1*} j_1^{-1}) \mathcal{F}_1
  \arrow[r,"="]
&
(i_{1*}j_{1*}) j_1^{-1} \mathcal{F}_1
  \arrow[r,"\simeq"]
&
i_{12*} (j_1^{-1} \mathcal{F}_1)
  \arrow[r,"="]
&
i_{12*} \mathcal{F}_{12}
\\
i_{2*}\mathcal{F}_2 
  \arrow[r,"i_{2*}(\eta_{2})"]
&
i_{2*}(j_{2*} j_2^{-1}) \mathcal{F}_2 
    \arrow[r,"="]
&
(i_{2*}j_{2*}) j_2^{-1} \mathcal{F}_2
  \arrow[r,"\simeq"]
&
i_{12*} (j_2^{-1} \mathcal{F}_2)
  \arrow[ur,"i_{12*}(\theta^{-1})"']
\end{tikzcd}
\end{center}

We now consider the following gluing functor $\mathfrak{G}$ and define $\mathcal F = i_{1*}\mathcal{F}_1  \times_{i_{12*}\mathcal{F}_{12}} i_{2*}\mathcal{F}_2$ via
\begin{align*}
\mathfrak G: \operatorname{Shv}_{\mathrm{hyp}}(K_1,\operatorname{Sp}) \times_{\operatorname{Shv}_{\mathrm{hyp}}(K_{12},\operatorname{Sp})} \operatorname{Shv}_{\mathrm{hyp}}(K_2,\operatorname{Sp}) &\longrightarrow \operatorname{Shv}_{\mathrm{hyp}}(K,\operatorname{Sp})\\
(\mathcal F_1, \mathcal F_2, \theta) &\longmapsto \mathcal F := i_{1*}\mathcal{F}_1  \times_{i_{12*}\mathcal{F}_{12}} i_{2*}\mathcal{F}_2.
\end{align*}

This construction is functorial in the gluing data. There is a projection diagram of $\mathcal F$ as follows:

\begin{center}
\begin{tikzcd}[column sep=normal, row sep=normal]
\mathcal{F}
  \arrow[r]
  \arrow[d]
&
i_{1*}\mathcal{F}_1
  \arrow[d]
\\
i_{2*}\mathcal{F}_2
  \arrow[r]
&
i_{12*}\mathcal{F}_{12}
\end{tikzcd}
\end{center}

Adjunction of $i_a^{-1} \dashv i_{a*}$ yields $i_1^{-1}\mathcal F\longrightarrow\mathcal F_1$ and $i_2^{-1}\mathcal F\longrightarrow\mathcal F_2$. The stalk functors preserve finite limits and detect equivalences between hypersheaves of spectra. We therefore check the following canonical maps stalkwise by this $\mathcal F_x\simeq (i_{1*}\mathcal F_1)_x \times_{(i_{12*}\mathcal F_{12})_x}(i_{2*}\mathcal F_2)_x$.

\textbf{1. $\mathfrak R \circ \mathfrak G \simeq \mathrm{id}_{\operatorname{Shv}_{\mathrm{hyp}}(K_1,\operatorname{Sp}) \times_{\operatorname{Shv}_{\mathrm{hyp}}(K_{12},\operatorname{Sp})} \operatorname{Shv}_{\mathrm{hyp}}(K_2,\operatorname{Sp}) }$}

Take arbitary gluing data $(\mathcal F_1,\mathcal F_2,\theta)$ and denote $\mathcal F = \mathfrak G(\mathcal F_1,\mathcal F_2,\theta)$. We will check the canonical equivalence $i_a^{-1}\mathcal F\longrightarrow\mathcal F_a$
\begin{itemize}
\item 
if $x \in K_1 \setminus K_2$, $\mathcal F_x\simeq (i_{1*}\mathcal F_1)_x \times_{0} 0 \simeq (\mathcal F_1)_x$.
\item
if $x \in K_2 \setminus K_1$, $\mathcal F_x\simeq 0 \times_{0} (i_{2*}\mathcal F_2)_x  \simeq (\mathcal F_2)_x$.
\item
if $x \in K_1 \cap K_2$, $\mathcal F_x\simeq(\mathcal F_1)_x\times_{(\mathcal F_{12})_x}(\mathcal F_2)_x$.
\end{itemize}

Both maps to $(\mathcal F_{12})_x$ are equivalences, so both projections are equivalences. The canonical maps $i_a^{-1}\mathcal F\to\mathcal F_a$ induce equivalences on every stalk and are therefore equivalences by hypercompleteness. The pullback construction ensures compatibility with $\theta$, and these comparisons are natural in the gluing data, so there is $\mathfrak R \circ \mathfrak G \simeq \mathrm{id}_{\operatorname{Shv}_{\mathrm{hyp}}(K_1,\operatorname{Sp}) \times_{\operatorname{Shv}_{\mathrm{hyp}}(K_{12},\operatorname{Sp})} \operatorname{Shv}_{\mathrm{hyp}}(K_2,\operatorname{Sp}) }$.

\textbf{2. $\mathfrak G \circ \mathfrak R \simeq \mathrm{id}_{\operatorname{Shv}_{\mathrm{hyp}}(K,\operatorname{Sp})}$}

Take an arbitary $\mathcal G \in \operatorname{Shv}_{\operatorname{hyp}}(K,\operatorname{Sp})$. Since the adjunction units and their compatibility with compositions of inclusions give a natural map
\begin{align*}
\mathcal G\longrightarrow
i_{1*}i_1^{-1}\mathcal G
\times_{i_{12*}i_{12}^{-1}\mathcal G}
i_{2*}i_2^{-1}\mathcal G
\simeq\mathfrak G(\mathfrak R(\mathcal G))
\end{align*}
we have
\begin{itemize}
\item 
if $x \in K_1 \setminus K_2$, $
\bigl(\mathfrak G(\mathfrak R(\mathcal G))\bigr)_x
\simeq  \mathcal G_x \times_{0} 0 \simeq \mathcal G_x$.
\item
if $x \in K_2 \setminus K_1$, $ \bigl(\mathfrak G(\mathfrak R(\mathcal G))\bigr)_x\simeq 0 \times_{0} \mathcal G_x  \simeq \mathcal G_x$.
\item
if $x \in K_1 \cap K_2$, $\bigl(\mathfrak G(\mathfrak R(\mathcal G))\bigr)_x \simeq  \mathcal G_x \times_{ \mathcal G_x}  \mathcal G_x\simeq \mathcal G_x$.
\end{itemize}
With these identifications, the canonical map $\mathcal G\to\mathfrak G(\mathfrak R(\mathcal G))$ induces an equivalence on every stalk. It is an equivalence by hypercompleteness. Since this comparison is natural in $\mathcal G$, we obtain $\mathrm{id}_{\operatorname{Shv}_{\mathrm{hyp}}(K,\operatorname{Sp})}\simeq\mathfrak G\circ\mathfrak R$, So $\mathfrak R$ and $\mathfrak G$ are inverse equivalences.
\end{proof}

\begin{remark}
\emph{Since the restriction functor $\mathfrak R$ is symmetric monoidal with respect to the componentwise symmetric monoidal structure on the pullback category and its underlying functor is an equivalence, it is a symmetric monoidal equivalence.}
\end{remark}

\begin{corollary}
\emph{The restriction functors induce the following canonical equivalence}
\begin{align*}
\mathcal M_X(K) \simeq \mathcal M_X(K_1) \times_{\mathcal M_X(K_{12})} \mathcal M_X(K_2).
\end{align*}
\end{corollary}

\begin{proof}
We use the notation from the preceding proof. Since $i_a^{-1}\mathbb O_K\simeq\mathbb O_{K_a}$ and $j_a^{-1}\mathbb O_{K_a}\simeq\mathbb O_{K_{12}}$, we can define 
\begin{align*}
\mathfrak{R}_{\operatorname{Mod}}: \mathcal M_X(K) \longrightarrow \mathcal{M}_X(K_1) \times_{\mathcal{M}_{X}(K_{12})} \mathcal M _X(K_2).
\end{align*}
An object in the target of $\mathfrak R_{\operatorname{Mod}}$ consists of modules $\mathcal F_{\operatorname{Mod}}\in\mathcal M_X(K_1)$ and $\mathcal G_{\operatorname{Mod}}\in\mathcal M_X(K_2)$ together with an $\mathbb O_{K_{12}}$-linear equivalence 
\begin{align*}
\theta_{\operatorname{Mod}}: j_{\operatorname{Mod},1}^{-1}\mathcal F_{\operatorname{Mod}}
\longrightarrow j_{\operatorname{Mod},2}^{-1}\mathcal G_{\operatorname{Mod}}.
\end{align*}
Set $\mathcal F_{\operatorname{Mod},12}:=j_{\operatorname{Mod},1}^{-1}\mathcal F_{\operatorname{Mod}}$, and denote the module pushforward functors by $i_{\operatorname{Mod},a*}$. Taking a pullback in $\mathcal M_X(K)$, we define
\begin{align*}
\mathfrak G_{\operatorname{Mod}}
(\mathcal F_{\operatorname{Mod}},
 \mathcal G_{\operatorname{Mod}},
 \theta_{\operatorname{Mod}})
:=
i_{\operatorname{Mod},1*}\mathcal F_{\operatorname{Mod}}
\times_{\,i_{\operatorname{Mod},12*}\mathcal F_{\operatorname{Mod},12}}
i_{\operatorname{Mod},2*}\mathcal G_{\operatorname{Mod}}.
\end{align*}

Consider the following diagram, where $U_{\operatorname{Mod},\times}$ and $U_{\operatorname{Mod},K}$ forget the module structures on objects of $\mathcal{M}_X(K_1)\times_{\mathcal{M}_X(K_{12})}\mathcal M_X(K_2)$ and $\mathcal M_X(K)$, respectively. These forgetful functors preserve pullbacks and are compatible with restriction and pushforward, so the underlying sheaf of spectra of this module gluing construction is precisely the gluing construction from the last proposition.

\begin{center}
\begin{tikzcd}[column sep=large, row sep=large]
\mathcal{M}_X(K_1)
\times_{\mathcal{M}_X(K_{12})}
\mathcal{M}_X(K_2)
  \arrow[r, shift left=0.7ex,
    "\mathfrak{G}_{\operatorname{Mod}}"]
  \arrow[d, "U_{\operatorname{Mod},\times}"']
&
\mathcal{M}_X(K)
  \arrow[l, shift left=0.7ex,
    "\mathfrak{R}_{\operatorname{Mod}}"]
  \arrow[d, "U_{\operatorname{Mod},K}"]
\\
\operatorname{Shv}_{\mathrm{hyp}}(K_1,\operatorname{Sp})
\times_{\operatorname{Shv}_{\mathrm{hyp}}(K_{12},\operatorname{Sp})}
\operatorname{Shv}_{\mathrm{hyp}}(K_2,\operatorname{Sp})
  \arrow[r, shift left=0.7ex, "\mathfrak{G}"]
&
\operatorname{Shv}_{\mathrm{hyp}}(K,\operatorname{Sp})
  \arrow[l, shift left=0.7ex, "\mathfrak{R}"]
\end{tikzcd}
\end{center}

The canonical comparison maps become equivalences after forgetting the module structures. Since the forgetful functors detect equivalences, $\mathfrak R_{\operatorname{Mod}}$ and $\mathfrak G_{\operatorname{Mod}}$ are inverse equivalences.
\end{proof}

\begin{proposition}
\emph{The restriction functors induce the following canonical equivalence}
\begin{align*}
\mathcal M_X^\alpha(K)
\simeq
\mathcal M_X^\alpha(K_1)
\times_{\mathcal M_X^\alpha(K_{12})}
\mathcal M_X^\alpha(K_2).
\end{align*}
\end{proposition}

\begin{proof}
Let $W$ be an open subset of $K$. The restriction gives a canonical functor
\begin{align*}
\mathcal M_X^\alpha(W)\longrightarrow
\mathcal M_X^\alpha(W\cap K_1)
\times_{\mathcal M_X^\alpha(W\cap K_{12})}
\mathcal M_X^\alpha(W\cap K_2).
\end{align*}

Choose an open cover $\{U_\lambda\}$ of $|X|$ and, for each $\lambda$, a trivialization of $\alpha|_{U_\lambda}$. Set $W_\lambda:=K\cap U_\lambda$ and $W_{\lambda_0\cdots\lambda_n}:=W_{\lambda_0}\cap \cdots \cap W_{\lambda_n}$. Then $\{W_\lambda\}$ is an open cover of $K$, and $\{W_\lambda\cap K_a\}_\lambda$ is an open cover of $K_a$. Since $W_{\lambda_0\cdots\lambda_n}\subseteq U_{\lambda_0}$, restricting the chosen trivialization of $\alpha|_{U_{\lambda_0}}$ identifies the twisted and untwisted module categories on $W_{\lambda_0\cdots\lambda_n}$ and $W_{\lambda_0\cdots\lambda_n}\cap K_a$, for $a=1,2,12$ (see Proposition \ref{MaX=MX}). Also, these equivalences are compatible with restriction. Since $W_{\lambda_0\cdots\lambda_n}\cap K_1$ and $W_{\lambda_0\cdots\lambda_n}\cap K_2$ form a closed cover of $W_{\lambda_0\cdots\lambda_n}$, the stalkwise argument of the preceding corollary shows that the untwisted restriction functor on this cover is an equivalence. Under the above trivializations, the twisted restriction functor is therefore also an equivalence.

The canonical restriction functors are compatible with further restriction, so the local equivalences give a termwise equivalence between the $\v{C}$ech diagrams. By descent for open covers of the hypersheaf of twisted module categories, we have
\begin{align*}
\mathcal M_X^\alpha(K)
\simeq
\lim_{[n]\in\Delta}
\prod_{\lambda_0,\ldots,\lambda_n}
\mathcal M_X^\alpha(W_{\lambda_0\cdots\lambda_n}).
\end{align*}

Similarly, for each $a=1,2,12$, we have
\begin{align*}
\mathcal M_X^\alpha(K_a)
\simeq
\lim_{[n]\in\Delta}
\prod_{\lambda_0,\ldots,\lambda_n}
\mathcal M_X^\alpha(W_{\lambda_0\cdots\lambda_n}\cap K_a).
\end{align*}

Taking limits of these local equivalences and using the fact that limits commute with limits, we get
\begin{align*}
\mathcal M_X^\alpha(K)
\simeq
\mathcal M_X^\alpha(K_1)
\times_{\mathcal M_X^\alpha(K_{12})}
\mathcal M_X^\alpha(K_2).
\end{align*}

This equivalence is induced by the canonical restriction functors and is independent of the local trivializations used to verify that they are equivalences.

\end{proof}

\subsection{Localization and Lifting of Compact Objects}

\begin{corollary}
\label{Autmorphism}
\emph{For any closed subset $K\subseteq |X|$, the global sections functor induces a canonical equivalence}
\begin{align*}
\Gamma(K,-):
\mathcal M_X(K)
\xrightarrow{\;\simeq\;}
\operatorname{Mod}_{\mathbb O_X(K)}(\operatorname{Sp}).
\end{align*}
\end{corollary}

\begin{proof}
We have already shown that $\mathbb O_X|_K$ is a compact generator (see Proposition \ref{O_X|_L-is-generator}) and that $\mathcal M_X(K)$ is presentable and stable (see Remark \ref{hypercomplete_of_Shv_(K,S)}). Since $\mathbb O_X|_K$ is a free module of rank $1$ over itself, for every $\mathcal F_{\operatorname{Mod}}\in\mathcal M_X(K)$ there is a natural equivalence of spectra
\begin{align*}
\operatorname{map}_{\mathcal M_X(K)} (\mathbb O_X|_K,\mathcal F_{\operatorname{Mod}}) \simeq \Gamma(K,\mathcal F_{\operatorname{Mod}}).
\end{align*}
Moreover, its endomorphism ring spectrum is given by
\begin{align*}
\operatorname{End}_{\mathcal M_X(K)}(\mathbb O_X|_K) \simeq \Gamma(K,\mathbb O_X|_K) \simeq \mathbb O_X(K).
\end{align*}
By \cite[Theorem 2.8]{AntieauGepner}, the mapping spectrum functor
$\operatorname{map}_{\mathcal M_X(K)}(\mathbb O_X|_K,-)$ gives an equivalence between $\mathcal M_X(K)$ and the category of modules over the opposite of the endomorphism ring spectrum of $\mathbb O_X|_K$. Since $\mathbb O_X(K)$ is a commutative ring spectrum, it is canonically equivalent to its opposite ring spectrum. Under these identifications, the above equivalence is given by the global sections functor $\Gamma(K,-)$.
\end{proof}

\begin{proposition}
\label{closed_restriction_localization}
\emph{Let $C\subseteq K\subseteq |X|$ be closed subsets, and let $j:C\hookrightarrow K$ denote the inclusion. The induced restriction functor}
\begin{align*}
j_{\operatorname{Mod}}^{-1}: \mathcal M_X(K) \longrightarrow \mathcal M_X(C)
\end{align*}
\emph{is a localization.}
\end{proposition}

\begin{proof}
Since $i_C=i_K\circ j$, we have $j^{-1}(\mathbb O_X|_K)\simeq\mathbb O_X|_C$. Thus restriction of sheaves of spectra sends $\mathbb O_X|_K$-modules to $\mathbb O_X|_C$-modules and defines the functor $j_{\operatorname{Mod}}^{-1}$.

For $\mathcal F_{\operatorname{Mod},C}\in\mathcal M_X(C)$, the underlying sheaf of spectra of its pushforward is defined by $\bigl(j_{\operatorname{Mod},*}\mathcal F_{\operatorname{Mod},C}\bigr)(W):=\mathcal F_{\operatorname{Mod},C}(W\cap C)$, where $W$ is an open subset of $K$. Its $\mathbb O_X|_K$-module structure is induced by the map $(\mathbb O_X|_K)(W) \longrightarrow (\mathbb O_X|_C)(W\cap C)$.

The adjunction $j^{-1}\dashv j_*$ is compatible with these module actions, so it induces an adjunction $j_{\operatorname{Mod}}^{-1}\dashv j_{\operatorname{Mod},*}$. Consider its counit
\[
\varepsilon_{\mathcal F}:
j_{\operatorname{Mod}}^{-1}j_{\operatorname{Mod},*}
\mathcal F_{\operatorname{Mod},C}
\longrightarrow
\mathcal F_{\operatorname{Mod},C}.
\]

For every $x\in C$, restriction does not change the stalk at $x$.
Since $C$ has the subspace topology, the sets $W\cap C$ form a neighborhood system of $x$ in $C$. Thus,
\begin{align*}
\bigl(j_{\operatorname{Mod}}^{-1}j_{\operatorname{Mod},*}
\mathcal F_{\operatorname{Mod},C}\bigr)_x
&\simeq
\operatorname*{colim}_{x\in W\subseteq K}
\mathcal F_{\operatorname{Mod},C}(W\cap C)\\
&\simeq
(\mathcal F_{\operatorname{Mod},C})_x.
\end{align*}
Under these identifications, $\varepsilon_{\mathcal F}$ induces the identity on every stalk.
By hypercompleteness and the conservativity of the forgetful functor, it is an equivalence.
Therefore, $j_{\operatorname{Mod},*}$ is fully faithful, so its left adjoint
$j_{\operatorname{Mod}}^{-1}$ is a localization functor.
\end{proof}

\begin{proposition}
\label{kernel_compact_generation}
\emph{The kernel of $j_{\operatorname{Mod}}^{-1}$ is generated by a family of objects that are compact in $\mathcal M_X(K)$}.
\end{proposition}

\begin{proof}
Since $j_{\operatorname{Mod}}^{-1}$ is a colimit-preserving functor between stable presentable $ \infty$ -categories, its kernel is stable and presentable, and colimits in the kernel can be computed in $ \mathcal M_X(K)$.

\textbf{1. Construction of compact objects in the kernel}

By the $t$-exactness of global sections, we have $\pi_0\mathbb O_X(K)\cong\Gamma\bigl(K,\pi_0(\mathbb O_X|_K)\bigr)$. For each $x\in K\setminus C$, a partition of unity subordinate to the open cover $\{K\setminus\{x\},K\setminus C\}$ gives functions $f_x,g_x\in\pi_0\mathbb O_X(K)$ such that $\operatorname{supp}(f_x)\subseteq K\setminus\{x\}$, $\operatorname{supp}(g_x)\subseteq K\setminus C$, and $f_x+g_x=1$. Define
\begin{align*}
\mathcal G_{\operatorname{Mod}}^{(x)} := \operatorname{cofib}\left( \mathbb O_X|_K \xrightarrow{f_x}\mathbb O_X|_K \right)\in\mathcal M_X(K).
\end{align*}

Since $\mathbb O_X|_K$ is compact and compact objects are closed under finite colimits, every $\mathcal G_{\operatorname{Mod}}^{(x)}$ is compact in $\mathcal M_X(K)$. Moreover, since $f_x$ restricts to $1$ on $C$ and
$j_{\operatorname{Mod}}^{-1}$ preserves colimits, there is
\begin{align*}
j_{\operatorname{Mod}}^{-1}\mathcal G_{\operatorname{Mod}}^{(x)} \simeq \operatorname{cofib}\left(\mathbb O_X|_C \xrightarrow{\mathrm{id}}\mathbb O_X|_C\right)\simeq0.
\end{align*}
Thus these objects lie in the kernel. Since filtered colimits in the kernel are computed in the original category, they are also compact as objects of the kernel.

\textbf{2. Generation of the kernel}

Let $\mathcal F_{\operatorname{Mod}} \in \ker(j_{\operatorname{Mod}}^{-1})$ and suppose that for every $x\in K\setminus C$, the mapping spectrum is 
\begin{align*}
\operatorname{map}_{\mathcal M_X(K)} \bigl(\mathcal G_{\operatorname{Mod}}^{(x)},\mathcal F_{\operatorname{Mod}}\bigr) \simeq 0.
\end{align*}
Applying $\operatorname{map}_{\mathcal M_X(K)}(-,\mathcal F_{\operatorname{Mod}})$ to the cofiber sequence defining $\mathcal G_{\operatorname{Mod}}^{(x)}$, we obtain
\begin{align*}
\operatorname{fib}\left(\Gamma(K,\mathcal F_{\operatorname{Mod}})\xrightarrow{f_x}\Gamma(K,\mathcal F_{\operatorname{Mod}})\right)\simeq0.
\end{align*}
By Corollary \ref{Autmorphism}, the endomorphism $f_x:\mathcal F_{\operatorname{Mod}}\to\mathcal F_{\operatorname{Mod}}$ is therefore an equivalence. Since $f_x$ vanishes on a neighborhood of $x$, we obtain $(\mathcal F_{\operatorname{Mod}})_x\simeq0$. The stalks at points of $C$ also vanish, since $j_{\operatorname{Mod}}^{-1}\mathcal F_{\operatorname{Mod}}\simeq0$. Thus all stalks of $\mathcal F_{\operatorname{Mod}}$ vanish. By hypercompleteness, $\mathcal F_{\operatorname{Mod}}\simeq0$.

From the criterion for compact generation in stable presentable $\infty$-categories \cite[Lemma 2.4]{AntieauGepner}, the family $\bigl\{\mathcal G_{\operatorname{Mod}}^{(x)}\bigr\}_{x\in K\setminus C}$
generates $\ker(j_{\operatorname{Mod}}^{-1})$.

\end{proof}

\begin{remark}
\emph{
We write $\mathcal G_{\operatorname{Mod}}^{(x)}$ for the cofiber to distinguish it from a stalk, which would be denoted by $\mathcal G_{\operatorname{Mod},x}$.}
\end{remark}

\begin{proposition}
\label{compact_lifting}
\emph{Let $C\subseteq K\subseteq |X|$ be closed subsets, and let $j:C\hookrightarrow K$ denote the inclusion. For every compact object $\mathcal F_{\operatorname{Mod},C}\in\mathcal M_X(C)$, there exists a compact object $\mathcal F_{\operatorname{Mod},K}\in\mathcal M_X(K)$ such that}
\begin{align*}
j_{\operatorname{Mod}}^{-1}\mathcal F_{\operatorname{Mod},K} \simeq \mathcal F_{\operatorname{Mod},C} \oplus \Sigma\mathcal F_{\operatorname{Mod},C}.
\end{align*}
\end{proposition}

\begin{proof}
The category $\mathcal M_X(K)$ is compactly generated, and $j_{\operatorname{Mod}}^{-1}$ is a colimit-preserving exact localization functor whose kernel is generated by a family of compact objects in $\mathcal M_X(K)$. From \cite[Theorem 2.1 and Corollary 0.9]{Neeman}, restriction preserves compact objects. Also, a compact object on $C$ is the restriction of a compact object on $K$ if and only if its $K_0$ class lies in the image of 
\begin{align*}
K_0\bigl(\mathcal M_X(K)^\omega\bigr)
\longrightarrow
K_0\bigl(\mathcal M_X(C)^\omega\bigr).
\end{align*}

Since $\mathcal F_{\operatorname{Mod},C}$ is compact, $\mathcal F_{\operatorname{Mod},C}\oplus\Sigma\mathcal F_{\operatorname{Mod},C}$ is also compact. Moreover, 
\begin{align*}
[\mathcal F_{\operatorname{Mod},C} \oplus\Sigma\mathcal F_{\operatorname{Mod},C}]
&= [\mathcal F_{\operatorname{Mod},C}] + [\Sigma\mathcal F_{\operatorname{Mod},C}] = [\mathcal F_{\operatorname{Mod},C}] -  [\mathcal F_{\operatorname{Mod},C}]  = 0.
\end{align*}

Since zero lies in the image of this homomorphism, there exists a compact object $\mathcal F_{\operatorname{Mod},K}\in\mathcal M_X(K)$ such that
\begin{align*}
j_{\operatorname{Mod}}^{-1}\mathcal F_{\operatorname{Mod},K}
\simeq
\mathcal F_{\operatorname{Mod},C}
\oplus\Sigma\mathcal F_{\operatorname{Mod},C}.
\end{align*}

\end{proof}

\begin{remark}
\emph{
With our notation, the quotient group in \cite[Corollary 0.9]{Neeman} is}
\begin{align*}
\frac{K_0(\mathcal M_X(C)^\omega)}{\operatorname{im}(K_0(\mathcal M_X(K)^\omega)\longrightarrow K_0(\mathcal M_X(C)^\omega))
}
\end{align*}
\emph{Here the homomorphism is induced by $j_{\operatorname{Mod}}^{-1}$. A class vanishes in this quotient precisely when it lies in the image of this homomorphism.}
\end{remark}

\subsection{Gluing Compact Generators}

\begin{proposition}
\emph{Let $K=K_1\cup K_2$ be a closed cover, and set $K_{12}:=K_1\cap K_2$. For an object $\mathcal F_{\operatorname{Mod},K}\in\mathcal M_X^\alpha(K)$, use the notation}
\begin{align*}
\mathcal F_{\operatorname{Mod},K}\in\mathcal M_X^\alpha(K),\quad \mathcal F_{\operatorname{Mod},a} := \mathcal F_{\operatorname{Mod},K}|_{K_a}, \quad a=1,2,12
\end{align*}
\emph{If $\mathcal F_{\operatorname{Mod},a}$ is compact for each $a=1,2,12$, then $\mathcal F_{\operatorname{Mod},K}$ is compact.}
\end{proposition}

\begin{proof}
For an arbitrary $\mathcal H_{\operatorname{Mod},K}\in\mathcal M_X^\alpha(K)$, set $\mathcal H_{\operatorname{Mod},a}:=\mathcal H_{\operatorname{Mod},K}|_{K_a}$ for $a=1,2,12$. The mapping spaces satisfy
\begin{align*}
&\operatorname{Map}_{\mathcal M_X^\alpha(K)}
\bigl(\mathcal F_{\operatorname{Mod},K},
      \mathcal H_{\operatorname{Mod},K}\bigr)\simeq\\
&\operatorname{Map}_{\mathcal M_X^\alpha(K_1)}\bigl(\mathcal F_{\operatorname{Mod},1}, \mathcal H_{\operatorname{Mod},1}\bigr)
\times_{\operatorname{Map}_{\mathcal M_X^\alpha(K_{12})}\left(\mathcal F_{\operatorname{Mod},12},\mathcal H_{\operatorname{Mod},12}\right)}   
\operatorname{Map}_{\mathcal M_X^\alpha(K_2)}\bigl(\mathcal F_{\operatorname{Mod},2},\mathcal H_{\operatorname{Mod},2}\bigr).
\end{align*}

The restriction functors preserve filtered colimits. This can be checked under local trivializations of $\alpha$ by using the same property for ordinary module sheaves. Since the three local objects are compact, the three mapping space functors on the right preserve filtered colimits. Filtered colimits commute with finite limits in spaces, so $\operatorname{Map}_{\mathcal M_X^\alpha(K)}(\mathcal F_{\operatorname{Mod},K},-)$ also preserves filtered colimits. Thus $\mathcal F_{\operatorname{Mod},K}$ is compact.

\end{proof}

\begin{construction}
\emph{We now construct gluing data for compact objects. Let $K_1,K_2\subseteq |X|$ be closed subsets, and set}
\begin{align*}
K=K_1\cup K_2,\quad K_{12}=K_1\cap K_2.
\end{align*}
\emph{Let $U_1,U_2\subseteq |X|$ be open subsets with $K_a\subseteq U_a$ for $a=1,2$, and choose a trivialization of $\alpha|_{U_a}$ for each $a$. Under the equivalence $\mathcal M_X^\alpha(K_1)\simeq\mathcal M_X(K_1)$ induced by the first trivialization, choose an object $\mathcal F_{\operatorname{Mod},1}$ corresponding to $\mathbb O_X|_{K_1}$. Its restriction $\mathcal F_{\operatorname{Mod},12}:=\mathcal F_{\operatorname{Mod},1}|_{K_{12}}$ corresponds to $\mathbb O_X|_{K_{12}}$ under the induced equivalence on $K_{12}$. In particular, $\mathcal F_{\operatorname{Mod},12}$ is compact in $\mathcal M_X^\alpha(K_{12})$.}

\emph{The chosen trivialization over $U_2$ similarly induces equivalences}
\begin{align*}
    \mathcal M_X^{\alpha}(K_2) \simeq \mathcal M_X(K_2), \quad     \mathcal M_X^{\alpha}(K_{12}) \simeq \mathcal M_X(K_{12}).
\end{align*}

\emph{Since equivalences of categories preserve compactness, $\mathcal F_{\operatorname{Mod},12}$ corresponds to a compact object under the second equivalence. Applying Proposition \ref{compact_lifting} to $j_2:K_{12}\hookrightarrow K_2$ under the trivialization over $U_2$, we obtain a compact object $\mathcal F_{\operatorname{Mod},2}\in\mathcal M_X^\alpha(K_2)$ together with an equivalence $\mathcal F_{\operatorname{Mod},2}|_{K_{12}}\simeq\mathcal F_{\operatorname{Mod},12}\oplus\Sigma\mathcal F_{\operatorname{Mod},12}$.}

\emph{Since compact objects are closed under shifts and finite direct sums, $\mathcal F_{\operatorname{Mod},1}\oplus\Sigma\mathcal F_{\operatorname{Mod},1}$ is compact in $\mathcal M_X^\alpha(K_1)$. Restriction preserves shifts and finite direct sums, so we obtain an equivalence}
\begin{align*}
\theta: \left( \mathcal F_{\operatorname{Mod},1} \oplus\Sigma\mathcal F_{\operatorname{Mod},1} \right)|_{K_{12}} &\simeq \mathcal F_{\operatorname{Mod},12} \oplus\Sigma\mathcal F_{\operatorname{Mod},12} \\
& \simeq  \mathcal F_{\operatorname{Mod},2}|_{K_{12}}.
\end{align*}
\emph{The triple $(\mathcal F_{\operatorname{Mod},1}\oplus\Sigma\mathcal F_{\operatorname{Mod},1},\mathcal F_{\operatorname{Mod},2},\theta)$ therefore defines gluing data for the closed cover $K=K_1\cup K_2$. By the closed-cover descent equivalence proved above, it determines an object $\mathcal F_{\operatorname{Mod},K}\in\mathcal M_X^\alpha(K)$ satisfying }
\begin{align*}
\mathcal F_{\operatorname{Mod},K}|_{K_1} \simeq
\mathcal F_{\operatorname{Mod},1}\oplus\Sigma\mathcal F_{\operatorname{Mod},1}, \quad\mathcal F_{\operatorname{Mod},K}|_{K_2}\simeq\mathcal F_{\operatorname{Mod},2}.
\end{align*}
\emph{These two local objects and their common restriction $\mathcal F_{\operatorname{Mod},12}\oplus\Sigma\mathcal F_{\operatorname{Mod},12}$ to the intersection are all compact.}

\end{construction}

\begin{remark}
\emph{The construction above is only the first step in constructing a global compact generator. It gives a global compact object whose restriction to $K_1$ is a generator. We will interchange $K_1$ and $K_2$, repeat the construction, and then take the direct sum of the two resulting global compact objects. This gives a global compact object whose restrictions to both pieces are compact generators. We will prove below that it is also a global generator.}
\end{remark}

\begin{construction}
\emph{Choose a compact generator $\mathcal G_{\operatorname{Mod},2}\in\mathcal M_X^\alpha(K_2)$. Repeating the construction with $K_1$ and $K_2$ interchanged gives a compact object $\mathcal G_{\operatorname{Mod},K}\in\mathcal M_X^\alpha(K)$ with}
\begin{align*}
\mathcal G_{\operatorname{Mod},K}|_{K_1}
\simeq\mathcal G_{\operatorname{Mod},1},\quad
\mathcal G_{\operatorname{Mod},K}|_{K_2}
\simeq
\mathcal G_{\operatorname{Mod},2}
\oplus\Sigma\mathcal G_{\operatorname{Mod},2}.
\end{align*}
\emph{We now define the global compact object as follows}
\begin{align*}
\widehat{\mathcal F}_{\operatorname{Mod},K} := \mathcal F_{\operatorname{Mod},K} \oplus \mathcal G_{\operatorname{Mod},K}.
\end{align*}
\emph{This object is compact, since finite direct sums preserve compactness. Its restrictions to $K_1$ and $K_2$ are}
\begin{align*}
\widehat{\mathcal F}_{\operatorname{Mod},K}|_{K_1} &\simeq
\mathcal F_{\operatorname{Mod},1} \oplus \Sigma \mathcal F_{\operatorname{Mod},1} \oplus \mathcal G_{\operatorname{Mod},1},\\
\widehat{\mathcal F}_{\operatorname{Mod},K}|_{K_2} &\simeq
\mathcal F_{\operatorname{Mod},2} \oplus \mathcal G_{\operatorname{Mod},2} \oplus \Sigma \mathcal G_{\operatorname{Mod},2}.
\end{align*}
\emph{Both restrictions are compact, since they are finite direct sums of compact objects. They generate their respective categories because they contain the compact generators $\mathcal F_{\operatorname{Mod},1}$ and $\mathcal G_{\operatorname{Mod},2}$, respectively, as direct summands.}
\end{construction}

\begin{proposition}
\emph{$\widehat{\mathcal F}_{\operatorname{Mod},K} = \mathcal F_{\operatorname{Mod},K} \oplus \mathcal G_{\operatorname{Mod},K}$ is a global compact generator in $\mathcal M_X^{\alpha}(K)$}.
\end{proposition}

\begin{proof}
Let $\mathcal H_{\operatorname{Mod},K}\in\mathcal M_X^\alpha(K)$ be an arbitrary object satisfying $\operatorname{map}_{\mathcal M_X^\alpha(K)} \bigl(\widehat{\mathcal F}_{\operatorname{Mod},K},\mathcal H_{\operatorname{Mod},K}\bigr) \simeq 0$. Since $\widehat{\mathcal F}_{\operatorname{Mod},K}$ is compact, the mapping spectrum functor preserves colimits. The tensor product also preserves colimits. By Proposition \ref{O_X|_L-is-generator}, the unit $\mathbb O_X|_K$ generates $\mathcal M_X(K)$. Hence the assumed vanishing extends to every $\mathcal P\in\mathcal M_X(K)$ by exactness and preservation of colimits,
\begin{align*}
\operatorname{map}_{\mathcal M_X^\alpha(K)} \left(
\widehat{\mathcal F}_{\operatorname{Mod},K}, \mathcal P\otimes_{\mathbb O_X|_K} \mathcal H_{\operatorname{Mod},K} \right) \simeq 0.
\end{align*}

For the closed inclusion $i_a:K_a\hookrightarrow K$, take $\mathcal P=i_{a*}(\mathbb O_X|_{K_a})$. The following equivalence can be checked on stalks
\begin{align*}
i_{a*}(\mathbb O_X|_{K_a}) \otimes_{\mathbb O_X|_K} \mathcal H_{\operatorname{Mod},K} \simeq i_{a*} (\mathcal H_{\operatorname{Mod},K}|_{K_a}).
\end{align*}
By the adjunction of pushforward and restriction, for $a=1,2$, there is
\begin{align*}
\operatorname{map}_{\mathcal M_X^\alpha(K_a)} \left( \widehat{\mathcal F}_{\operatorname{Mod},K}|_{K_a}, \mathcal H_{\operatorname{Mod},K}|_{K_a} \right) \simeq 0.
\end{align*}
Since the restrictions of $\widehat{\mathcal F}_{\operatorname{Mod},K}$ to $K_1$ and $K_2$ are generators, these vanishings imply $\mathcal H_{\operatorname{Mod},K}|_{K_1}\simeq0$ and $\mathcal H_{\operatorname{Mod},K}|_{K_2}\simeq0$. By closed cover descent, $\mathcal H_{\operatorname{Mod},K}\simeq 0$.

\end{proof}

\begin{corollary}
\emph{For each $a=1,2,12$, let $i_a:K_a\hookrightarrow K$ denote the closed inclusion, and let $i_{a*}$ denote the corresponding pushforward functor for twisted module sheaves. With this notation, the global compact generator in $\mathcal M_X^\alpha(K)$ can be written as follows}
\begin{align*}
\widehat{\mathcal F}_{\operatorname{Mod},K} \simeq
&\left[i_{1*}\left(\mathcal F_{\operatorname{Mod},1}
\oplus\Sigma\mathcal F_{\operatorname{Mod},1}
\right) \times_{i_{12*}\left( \mathcal F_{\operatorname{Mod},12} \oplus \Sigma \mathcal F_{\operatorname{Mod},12} \right)}
i_{2*}\mathcal F_{\operatorname{Mod},2}\right] \\
&\oplus \left[ i_{1*}\mathcal G_{\operatorname{Mod},1} \times_{ i_{12*}\left( \mathcal G_{\operatorname{Mod},12} \oplus \Sigma \mathcal G_{\operatorname{Mod},12} \right)}
i_{2*}\left(\mathcal G_{\operatorname{Mod},2} \oplus \Sigma \mathcal G_{\operatorname{Mod},2} \right) \right].
\end{align*}
\end{corollary}

We have considered the case of a cover by two pieces. We now use induction to handle finite closed covers. Proposition \ref{compact_lifting} applies to untwisted module categories, so its use above relied on a chosen trivialization. We next prove a version that applies directly to twisted module categories.

\begin{proposition}
\label{compact_lifting_with_torsor}
\emph{Let $L \subseteq  T \subseteq |X| $ be closed subsets and let $j:L\hookrightarrow T$ denote the inclusion. We suppose that $\mathcal M_X^\alpha(T)$ has a compact generator. Then for every compact object $\mathcal U_{\operatorname{Mod},L}\in\mathcal M_X^\alpha(L)$, there exists a compact object $\mathcal U_{\operatorname{Mod},T}\in\mathcal M_X^\alpha(T)$ such that}
\begin{align*}
j_{\operatorname{Mod}}^{-1}\mathcal U_{\operatorname{Mod},T}
\simeq
\mathcal U_{\operatorname{Mod},L}
\oplus\Sigma\mathcal U_{\operatorname{Mod},L}.
\end{align*}
\end{proposition}

\begin{proof}
Fix a compact generator $\mathcal P_{\operatorname{Mod},T}\in\mathcal M_X^\alpha(T)$. After replacing $(K,C)$ by $(T,L)$, the localization constructed in Proposition \ref{closed_restriction_localization} extends to the twisted categories by local descent, since module pushforward is compatible with the invertible-module transition functors. For the kernel generation argument (Proposition \ref{kernel_compact_generation}), replace the unit object used to construct the kernel generators by $\mathcal P_{\operatorname{Mod},T}$, and replace the global sections functor by $\operatorname{map}_{\mathcal M_X^\alpha(T)} (\mathcal P_{\operatorname{Mod},T},-).$

This functor is exact and conservative because $\mathcal P_{\operatorname{Mod},T}$ is a generator. The functions $f_x\in\pi_0\mathbb O_X(T)$, obtained from partitions of unity, act through the $\mathcal M_X(T)$-module structure on $\mathcal M_X^\alpha(T)$. The same cofiber argument therefore shows that the kernel is generated by objects compact in the source. The Neeman localization and $K_0$ argument used in Proposition \ref{compact_lifting} now applies and gives the required compact lift.
\end{proof}

\begin{remark}
\emph{The preceding construction for two pieces also applies whenever both $\mathcal M_X^\alpha(K_1)$ and $\mathcal M_X^\alpha(K_2)$ have compact generators. The twists $\alpha|_{K_1}$ and $\alpha|_{K_2}$ need not be trivial. The preceding proof also shows that restriction to a closed subset preserves compact objects. Proposition \ref{compact_lifting_with_torsor} allows us to lift and match compact objects on the intersection. The gluing construction and the arguments for compactness and generation are the same as before.}
\end{remark}

\begin{proposition}
\emph{For every Picard twist $\alpha$, the category $\mathcal M_X^\alpha(|X|)$ admits a compact generator $\widehat{\mathcal F}_{\operatorname{Mod},|X|}$.}
\end{proposition}

\begin{proof}
Choose a finite closed cover $|X|=K_1\cup\cdots\cup K_n$ with $K_a\subseteq U_{i(a)}$, where a trivialization of $\alpha|_{U_{i(a)}}$ has been chosen for each $a$. Each category $\mathcal M_X^\alpha(K_a)$ therefore has a compact generator. Set $T_r:=K_1\cup\cdots\cup K_r$. The case $r=1$ is established. Suppose that $\mathcal M_X^\alpha(T_{r-1})$ has a compact generator. Since $T_r=T_{r-1}\cup K_r$, the two-piece construction gives a compact generator of $\mathcal M_X^\alpha(T_r)$. Induction completes the proof.
\end{proof}

\begin{remark}
\label{global_resrtiction_to_local}
\emph{The restriction of the global compact generator to any closed subset $K\subseteq|X|$ is again a compact generator. Compactness follows from Proposition \ref{compact_lifting_with_torsor}. To prove generation, let $i:K\hookrightarrow|X|$ denote the inclusion. By adjunction, we have}
\begin{align*}
\operatorname{map}\left( \widehat{\mathcal F}_{\operatorname{Mod},|X|}|_K,\mathcal H \right) \simeq
\operatorname{map}\left( \widehat{\mathcal F}_{\operatorname{Mod}, |X|},i_*\mathcal H \right).
\end{align*}
\emph{If the left side vanishes, the generator property of $\widehat{\mathcal F}_{\operatorname{Mod},|X|}$ implies $i_*\mathcal H\simeq0$. Since $i_*$ is fully faithful, $\mathcal H\simeq0$.}
\end{remark}

\begin{corollary}
\emph{Set $L_a:=\bigcup_{b\ne a}K_b$. Choose a compact generator $\mathcal P_a\in\mathcal M_X^\alpha(K_a)$ and a compact object $\mathcal Q_a\in\mathcal M_X^\alpha(L_a)$ equipped with an equivalence $\mathcal Q_a|_{K_a\cap L_a}\simeq(\mathcal P_a\oplus\Sigma\mathcal P_a)|_{K_a\cap L_a}$. Then a global compact generator can be chosen as follows:}
\begin{align*}
\widehat{\mathcal F}_{\operatorname{Mod},|X|} &:=
\bigoplus_{a=1}^{n}
\left[
i_{K_a,*}
(\mathcal P_a\oplus\Sigma\mathcal P_a)
\underset{
i_{K_a\cap(\bigcup_{b\ne a}K_b),*}
\left(
(\mathcal P_a\oplus\Sigma\mathcal P_a)
\big|_{K_a\cap(\bigcup_{b\ne a}K_b)}
\right)
}{\times}
i_{\bigcup_{b\ne a}K_b,*}\mathcal Q_a
\right]\\
&\simeq
\bigoplus_{a=1}^{n}
\operatorname*{lim}_{[r]\in\Delta}
\prod_{b_0,\ldots,b_r=1}^{n}
i_{K_{b_0}\cap\cdots\cap K_{b_r},*}
\left(
\mathcal E_{a,b_0}
\big|_{K_{b_0}\cap\cdots\cap K_{b_r}}
\right), \quad 
\mathcal E_{a,b}=
\begin{cases}
\mathcal P_a \oplus \Sigma \mathcal P_a, &b=a\\
\mathcal Q_a|_{K_b}, &b\ne a
\end{cases}
\end{align*}
\end{corollary}

\section{Azumaya Representatives of Picard Twists}

Following \cite[Definition 1.1]{BakerRichterSzymik}, we define derived Azumaya algebras as follows.

\begin{definition}
\emph{Let $\mathcal A\in\operatorname{Alg}_{\mathbb E_1}(\mathcal M_X(|X|))$ be a sheaf of $\mathbb E_1$-algebras over $\mathbb O_X$. We call $\mathcal A$ a derived Azumaya algebra if it satisfies the following conditions.}

\begin{enumerate}
\item \emph{\textbf{Local perfectness:} Every point has an open neighborhood $U$ such that the underlying module $\mathcal A|_U$ can be constructed from $\mathbb O_X|_U$ by finitely many shifts, finite colimits, and taking direct summands.}
\item \emph{\textbf{Faithfulness:} For every $\mathcal H\in\mathcal M_X(|X|)$, we require}
\begin{align*}
\mathcal A\otimes_{\mathbb O_X}\mathcal H \simeq 0
\quad \Longrightarrow \quad \mathcal H \simeq 0.
\end{align*}
\item \emph{\textbf{Azumaya condition:} The natural morphism}
\begin{align*}
\mu_{\mathcal A}:\quad
\mathcal A\otimes_{\mathbb O_X}\mathcal A^{\mathrm{op}}
&\longrightarrow
\underline{\operatorname{End}}_{\mathbb O_X}(\mathcal A)\\
a\otimes b &\longmapsto (c\longmapsto acb)
\end{align*}
\emph{is an equivalence.}
\end{enumerate}

\end{definition}

\begin{definition}
\emph{For a sheaf $\mathcal A$ of derived Azumaya algebras, its hypersheaf of left module categories is given by}
\begin{align*}
V\longmapsto
\operatorname{LMod}_{\mathcal A|_V}
\bigl(\mathcal M_X(V)\bigr),\bigl(\mathcal M_X(V)\bigr),\quad
V\subseteq |X|\text{ \emph{open}}.
\end{align*}
\emph{We say that two sheaves of derived Azumaya algebras $\mathcal A$ and $\mathcal B$ are Morita equivalent if their hypersheaves of left module categories are $\mathcal M_X$-linearly equivalent. We say that $\mathcal A$ is locally Morita trivial in the open topology if there exists an open cover $\{U_i\}$ such that each $\mathcal A|_{U_i}$ is Morita equivalent to $\mathbb O_X|_{U_i}$. We denote by $\mathbf{dBr}^{\mathrm{open}}(X)$ the group of Morita equivalence classes of these sheaves of algebras, with group operation induced by $\otimes_{\mathbb O_X}$}.
\end{definition}

\begin{construction}
\emph{Given a Picard twist $\alpha\in\Gamma(|X|,\mathrm B\operatorname{Pic}_{\mathbb O_X})$ and a global compact generator $\widehat{\mathcal F}_{\operatorname{Mod},|X|}
\in\mathcal M_X^\alpha(|X|)$, we define} $\mathcal A_\alpha := \underline{\operatorname{End}}_{\alpha} \bigl(\widehat{\mathcal F}_{\operatorname{Mod},|X|}\bigr)^{\mathrm{op}}$. \emph{Here, the internal endomorphism algebra sheaf is given by the following mapping spectra: for each open subset $V\subseteq|X|$, there is}
\begin{align*}
\mathcal A_\alpha(V) = \operatorname{map}_{\mathcal M_X^\alpha(V)} \left(\widehat{\mathcal F}_{\operatorname{Mod},|X|}|_V,\widehat{\mathcal F}_{\operatorname{Mod},|X|}|_V \right)^{\mathrm{op}}.
\end{align*}
\emph{The $\mathbb O_X$-linear structure of the twisted module categories gives these mapping spectra compatible $\mathbb O_X(V)$-module structures. By hyperdescent of $\mathcal M_X^\alpha$, they form a hypersheaf. Composition and identity morphisms define its multiplication and unit, so $\mathcal A_\alpha\in
\operatorname{Alg}_{\mathbb E_1} \bigl(\mathcal M_X(|X|)\bigr)$. We take the opposite algebra here so that the internal Hom functor below takes values in left $\mathcal A_\alpha$-modules.}

\end{construction}

\begin{proposition}
\emph{The sheaf $\mathcal A_\alpha$ constructed above is a derived Azumaya algebra.}
\end{proposition}

\begin{proof}
\textbf{1. Local perfectness:}
Fix $x\in|X|$. Choose an open neighborhood $U$ of $x$ over which $\alpha$ is trivial, fix a trivialization, and choose a closed neighborhood $K$ of $x$ contained in $U$. Set $W:=\operatorname{int}_{|X|}(K)$. Under the induced equivalence $\mathcal M_X^\alpha(K)\simeq\mathcal M_X(K)$, the restriction of the global generator corresponds to a compact generator $\mathcal P_K$. Since the unit compactly generates $\mathcal M_X(K)$, the object $\mathcal P_K$ can be constructed from the unit by finitely many shifts, finite colimits, and taking direct summands. Hence $\mathcal P_K|_W$ is perfect. As an $\mathbb O_X|_W$-module,
\begin{align*}
\mathcal A_\alpha|_W \simeq
(\mathcal P_K|_W)^\vee \otimes_{\mathbb O_X|_W} (\mathcal P_K|_W).
\end{align*}
Also, duality and tensor products preserve perfectness, so $\mathcal A_\alpha$ is locally perfect.

\textbf{2. Faithfulness:}
The unit $\mathbb O_X|_K$ can be constructed from $\mathcal P_K$ by finite operations. Passing to the stalk at $x$, we see that $\mathbb O_{X,x}$ can likewise be constructed from $(\mathcal P_K)_x$. Thus $(\mathcal P_K)_x$ is faithful, perfect, and dualizable. By \cite[Lemma 2.1(1)]{BakerRichterSzymik}, its dual is also faithful. Hence $(\mathcal A_\alpha)_x\simeq(\mathcal P_K)_x^\vee\otimes_{\mathbb O_{X,x}}(\mathcal P_K)_x$ is faithful.

Suppose that $\mathcal A_\alpha\otimes_{\mathbb O_X}\mathcal H\simeq0$. Taking stalks gives $(\mathcal A_\alpha)_x\otimes_{\mathbb O_{X,x}}\mathcal H_x\simeq0$. By faithfulness of $(\mathcal A_\alpha)_x$, we obtain $\mathcal H_x\simeq0$ for every $x$. Hypercompleteness then implies $\mathcal H\simeq0$. Thus $\mathcal A_\alpha$ is faithful.

\textbf{3. Azumaya Condition:}
For each $x$, we have $(\mathcal A_\alpha)_x\simeq\operatorname{End}_{\mathbb O_{X,x}}((\mathcal P_K)_x)^{\mathrm{op}}$. By \cite[Proposition 2.2]{BakerRichterSzymik} and invariance of the Azumaya property under passage to opposite algebras, the natural morphism below is an equivalence 
\begin{align*}
\mu_{\mathcal A_{\alpha},x}:(\mathcal A_\alpha)_x \otimes_{\mathbb O_{X,x}}
(\mathcal A_\alpha)_x^{\mathrm{op}} \longrightarrow \operatorname{End}_{\mathbb O_{X,x}} \bigl((\mathcal A_\alpha)_x\bigr).
\end{align*}
Since $\mathcal A_\alpha$ is locally perfect, taking internal endomorphisms commutes with taking stalks. Thus $\left(\underline{\operatorname{End}}_{\mathbb O_X} (\mathcal A_\alpha)\right)_x \simeq \operatorname{End}_{\mathbb O_{X,x}}
\bigl((\mathcal A_\alpha)_x\bigr)$, so the above $\mu_{\mathcal A_\alpha,x}$ is the stalk of the global natural map $\mu_{\mathcal A_\alpha}$. By hypercompleteness, $\mu_{\mathcal A_\alpha}$ is an equivalence.

\end{proof}

\begin{proposition}
\emph{For every open subset $V\subseteq|X|$, the internal Hom functor gives an equivalence:}
\begin{align*}
\mathcal M_X^\alpha(V)
&\longrightarrow \operatorname{LMod}_{\mathcal A_\alpha|_V} \bigl(\mathcal M_X(V)\bigr)\\
\mathcal H &\longmapsto \underline{\operatorname{Hom}}_\alpha \left( \widehat{\mathcal F}_{\operatorname{Mod},|X|}|_V, \mathcal H \right)
\end{align*}
\emph{These equivalences are $\mathcal M_X(V)$-linear and compatible with restriction to open subsets, so $\mathcal A_\alpha$ represents $\alpha$}.
\end{proposition}

\begin{proof}
$\widehat{\mathcal F}_{\operatorname{Mod},|X|}|_{V}$ has a natural right $\mathcal{A}_{\alpha}|_{V}$-module structure, so there is an adjunction:
\begin{align*}
\widehat{\mathcal F}_{\operatorname{Mod},|X|}|_V \otimes_{\mathcal A_{\alpha}|_{V}} - :
\operatorname{LMod}_{\mathcal A_{\alpha}|_V} \bigl(\mathcal{M}_{X}(V)\bigr)&\rightleftarrows \mathcal{M}_X^{\alpha}(V) : \underline{\operatorname{Hom}}_\alpha \left( \widehat{\mathcal F}_{\operatorname{Mod},|X|}|_V, - \right)
\end{align*}
For each $x\in V$, we have already shown that $(\mathcal P_K)_x$ is a compact generator of $\operatorname{Mod}_{\mathbb O_{X,x}}$ and that $(\mathcal A_\alpha)_x\simeq\operatorname{End}_{\mathbb O_{X,x}}((\mathcal P_K)_x)^{\mathrm{op}}$. By local perfectness, the adjunction is compatible with taking stalks. By \cite[Theorem 2.8]{AntieauGepner}, the induced adjunction on each stalk is a Morita equivalence. Since the unit and counit of the adjunction are equivalences on every stalk, hypercompleteness gives the required equivalence. Under a local trivialization $\mathcal M_X^\alpha(W)\simeq\mathcal M_X(W)$, the restricted generator $\widehat{\mathcal F}_{\operatorname{Mod},|X|}|_W$ corresponds to $\mathcal P_K|_W$. The following identification shows that the internal Hom functor is linear and compatible with restriction:
\begin{align*}
\underline{\operatorname{Hom}}_{\mathbb O_{X}|_W} \left( \mathcal P_K|_{W},- \right)  \simeq
(\mathcal P_K|_W)^\vee \otimes_{\mathbb O_X|_W} -
\end{align*}
This shows that these equivalences are $\mathcal M_X(V)$-linear and compatible with restriction to open subsets, so $\mathcal A_\alpha$ represents $\alpha$.

\end{proof}

\begin{remark}
\emph{The Morita equivalence class of $\mathcal A_\alpha$ is independent of the choices of closed cover, local trivializations, compact lifts, and global compact generator. Indeed, for any two resulting algebra sheaves $\mathcal A_\alpha$ and $\mathcal A'_\alpha$ and every open subset $V$, we have}
\begin{align*}
 \operatorname{LMod}_{\mathcal A_\alpha|_V} \bigl(\mathcal M_X(V)\bigr) \simeq \mathcal{M}_{X}^{\alpha}(V) \simeq \operatorname{LMod}_{\mathcal A'_\alpha|_V} \bigl(\mathcal M_X(V)\bigr).
\end{align*}
\emph{These equivalences are $\mathcal M_X(V)$-linear and compatible with restriction, so $\mathcal A_\alpha$ and $\mathcal A'_\alpha$ are Morita equivalent.}
\end{remark}

\begin{proposition}
\emph{Let $X$ be a compact derived smooth manifold. Then there exists a bijection}
\begin{align*}
\Phi_X: \mathbf{dBr}^{\mathrm{open}}(X)
\longrightarrow
 \pi_0 \Gamma \bigl(|X|,\mathrm{B}\operatorname{Pic}_{\mathbb O_X}\bigr).    
\end{align*}
\end{proposition}

\begin{proof}
By \cite[Theorem~4.8.4.1]{LurieHA}, we know that the linear transition autoequivalences between local Morita trivializations are given by tensoring with invertible modules.  Their coherent descent data define a Picard twist whose equivalence class is independent of the chosen cover and trivializations and depends only on the Morita class of the algebra sheaf. Let $\Phi_X$ send the Morita class of the algebra sheaf to the equivalence class of this twist.

\textbf{1. $\Phi_X$ is injective.}

If $\Phi_X([\mathcal A])=\Phi_X([\mathcal B])=[\alpha]$, both hypersheaves of left module categories are linearly equivalent to $\mathcal M_X^\alpha$, and these equivalences are compatible with restriction to open subsets. So $\mathcal A$ and $\mathcal B$ are Morita equivalent, so $[\mathcal A]=[\mathcal B]$.

\textbf{2. $\Phi_X$ is surjective.}

For any Picard twist $\alpha$, the algebra $\mathcal A_\alpha$ constructed above is locally Morita trivial in the open topology, and its hypersheaf of left module categories is linearly equivalent to $\mathcal M_X^\alpha$.
\end{proof}

\begin{corollary}
\emph{Let $X$ be a compact derived smooth manifold. Then there exists a bijection:}
\begin{align*}
\mathbf{dBr}^{\mathrm{open}}(X)
\cong
H^1(|X|;\underline{\mathbb Z})
\times
\pi_0\Gamma\bigl(|X|,\mathrm{B}^2 GL_1(\mathbb{O}_X)\bigr)
\end{align*}
\end{corollary}

\section{Brauer Groups of Ordinary Smooth Manifolds}
We have established a classification theory over derived smooth manifolds. We now apply it to ordinary smooth manifolds with coefficient sheaves $C^\infty_X(\mathbb C)$ and $C^\infty_X(\mathbb R)$. In the complex case, we recover the degree-three integral invariant; in the real case, we recover the degree-two $\mathbb Z/2$-invariant within the open locally Morita trivial theory.

\begin{proposition}
\emph{Let $X$ be a compact ordinary smooth manifold, and set $\mathbb O_X:=HC^\infty_X(\mathbb C)$. Then there exists a bijection:}
\begin{align*}
\mathbf{dBr}^{\mathrm{open}}(X) \cong H^1(|X| ; \underline{\mathbb Z})
\times H^3(|X| ;\underline{\mathbb Z})
\end{align*}
\end{proposition}

\begin{proof}

The preceding classification argument also applies to $HC^\infty_X(\mathbb C)$, since its stalks are connective local ring spectra and partitions of unity give the same exactness and generation properties. Since $\mathbb O_X$ comes from the ordinary sheaf of functions, there is $GL_1(\mathbb O_X) \simeq C^\infty_X(\mathbb C)^\times$, so $\pi_0\Gamma\bigl(|X|,\mathrm{B}^2 GL_1(\mathbb{O}_X)\bigr) \cong H^2(|X|;C^\infty_X(\mathbb C)^\times)$. Now consider the exponential exact sequence of sheaves of smooth functions:
\begin{align*}
0 \longrightarrow \underline{\mathbb Z} \longrightarrow C^\infty_X(\mathbb C) \xrightarrow{f\mapsto e^{2\pi i f}} C^\infty_X(\mathbb C)^\times \longrightarrow 1
\end{align*}
Since $C^\infty_X(\mathbb C)$ is a fine sheaf, its cohomology vanishes in positive degrees, so there is $H^2\left( |X| ; C^\infty_X(\mathbb C)^\times\right) \cong H^3(|X| ; \underline{\mathbb Z})$.
\end{proof}

\begin{remark}
\emph{For an ordinary smooth manifold, sheaf cohomology of the constant sheaf is naturally isomorphic to singular cohomology with integer coefficients, namely $H^k_{\mathrm{sh}}(|X|;\underline{\mathbb Z})
\cong H^k_{\mathrm{sing}}(|X|;\mathbb Z)$. The classical Dixmier–Douady theory classifies locally trivial $C^*$-algebra bundles with fibers given by the algebra of compact operators by $H^3_{\mathrm{sing}}(|X|;\mathbb Z)$ (see \cite[Theorem 11]{DixmierDouady}). The $H^1$ factor in our theory records shift twists. When this component is zero, or when we ignore it, we obtain the same cohomological classification as in the classical theory. If we only consider algebra bundles with fibers given by finite-dimensional complex matrix $C^*$-algebras, their Brauer group corresponds to the torsion subgroup of $H^3$ (see \cite[Theorem 2.15(ii)]{DadarlatPennigII}, with $D=\mathbb C$). The derived Azumaya representatives in our theory also allow non-torsion classes in degree three.}
\end{remark}

\begin{proposition}
\emph{Let $X$ be a compact ordinary smooth manifold, and set $\mathbb O_X:=HC^\infty_X(\mathbb R)$. Then there exists a bijection:}
\begin{align*}
\mathbf{dBr}^{\mathrm{open}}(X) \cong H^1(|X| ; \underline{\mathbb Z})
\times H^2(|X| ;\underline{\mathbb Z/2})
\end{align*}
\end{proposition}

\begin{proof}
Since $\mathbb O_X$ comes from the ordinary sheaf of functions, there is $GL_1(\mathbb O_X) \simeq C^\infty_X(\mathbb R)^\times$, so $\pi_0\Gamma\bigl(|X|,\mathrm{B}^2 GL_1(\mathbb{O}_X)\bigr) \cong H^2(|X|; C^\infty_X(\mathbb R)^\times)$. Now consider the isomorphism of sheaves of abelian groups $C^\infty_X(\mathbb R)^\times
\cong \underline{\mathbb Z/2}\times C^\infty_X(\mathbb R)$.
Since $C^\infty_X(\mathbb R)$ is a fine sheaf, its cohomology vanishes in positive degrees, so there is $H^2 \left(|X|; C^\infty_X(\mathbb R)^\times\right) \cong H^2(|X|;\underline{\mathbb Z/2})$.
\end{proof}

\begin{remark}
\emph{Donovan–Karoubi’s real graded Brauer theory contains, besides $H^2(|X|;\mathbb Z/2)$, local types with values in $\mathbb Z/8$ and parity grading twists in $H^1(|X|;\mathbb Z/2)$ (see \cite[Theorem 6]{DonovanKaroubi}). We require the algebras to be locally Morita equivalent to $\mathbb O_X$ on an open cover. So this excludes nontrivial local real Brauer types, such as quaternion algebras. We also use integer cohomological shifts without introducing a separate $\mathbb Z/2$-grading, so that we recover the classical $H^2(|X|;\mathbb Z/2)$ invariant here.}

\emph{It is interesting and worth noting that classical graded Brauer theory over $\mathbb C$ also contains local types in $H^0(|X|;\mathbb Z/2)$ and parity twists in $H^1(|X|;\mathbb Z/2)$ (see \cite[Theorem 11]{DonovanKaroubi}). These graded invariants also show the difference between its $\mathbb Z/2$-graded framework and our framework of integer shifts.}
\end{remark}

\section{Future Work}
\subsection{Brauer Theory in Higher Differential Geometry and Derived Analytic Geometry}
We temporarily call the approach of defining the module category through the underlying (simplicial) commutative $\mathbb R$-algebra $A^{\mathrm{alg}}$ of a (simplicial) $C^\infty$-ring $A$ “linearization”. For example, when defining modules over $C^\infty$-rings, Joyce defines a module over a $C^\infty$-ring $A$ to be a module over its underlying commutative $\mathbb R$-algebra $A^{\mathrm{alg}}$ \cite[Definition 2.13]{Joyce}. Steffens also defines module sheaves over a sheaf of simplicial $C^\infty$-rings through its underlying sheaf of commutative $\mathbb R$-algebras \cite[§3.2.3]{Steffens}. We follow the same idea: we first take the underlying sheaf of simplicial commutative $\mathbb R$-algebras of the structure sheaf, then pass to spectra, and study modules over the resulting sheaf of $\mathbb E_\infty$-algebras.

This raises the question: how much of the information encoded by the original smooth operations does this “linearization” retain? Does the choice of module theory change Picard objects, Morita theory, or the Brauer classification? We do not assume that smooth information is lost. We hope to study this by comparing different module theories.

Higher differential geometry and derived analytic geometry give a future direction for this study. Ben-Bassat, Kelly, and Kremnizer give fully faithful functors from ordinary and simplicial $C^\infty$-rings to the corresponding categories of analytic algebras \cite[§9.2.5]{BenBassatKellyKremnizer}, which we may use to encode the smooth information. These functors preserve coproducts and give a basis for studying bornological modules and completed tensor products. Savage’s work on $C^\infty$-bornological rings and their use in proving representability of derived moduli stacks of solutions to nonlinear elliptic PDEs provides a further reference for this direction \cite{Savage}.

\subsection{Graded Derived Brauer Theory}

Another direction is to develop a graded derived Brauer theory. Besides the integer cohomological shifts used in this paper, we can add a separate $\mathbb Z/2$-grading and study graded Azumaya algebra sheaves, graded tensor products, and Morita equivalences. To compare more fully with Donovan–Karoubi’s classical graded Brauer theory, we also need to relax the requirement of local Morita triviality on open subsets and allow more local graded Morita types.

\section{Acknowledgement}

We thank Christoph Schweigert for putting us in contact with Pelle Steffens. We are particularly grateful to Pelle Steffens for his detailed feedback on an earlier version and for suggesting the use of closed covers to glue local compact generators. We also thank Rhiannon Savage for helpful correspondence on bornological algebra, and Jurek Bleischwitz for many helpful discussions and suggestions, especially concerning the use of bornological methods in higher differential geometry and possible further developments in derived analytic geometry. Y. M. also thanks Ulrich Pennig for pleasant conversations at the conference “Interactions between operator algebras, K-theory and homotopy theory” in M{\"u}nster.

\end{document}